\documentclass{article}

\usepackage[T1]{fontenc}

\usepackage{amsmath}
\usepackage{amssymb}

\usepackage[dvipsnames]{xcolor}
\usepackage[colorlinks=true,
            linktocpage=true,
            citecolor=RubineRed,
            linkcolor=RoyalBlue]{hyperref}
\usepackage[numbers]{natbib}
\usepackage[a4paper,margin=2.5cm]{geometry}
\usepackage[noabbrev,capitalise]{cleveref}
\usepackage{doi}

\usepackage[dvipsnames]{xcolor}
\usepackage{stmaryrd}
\usepackage[nointegrals]{wasysym}  
\usepackage{fontawesome}  
\usepackage[all]{xy}
\usepackage{tikz}
  \usetikzlibrary{arrows,arrows.meta,matrix}

\usepackage{pifont}

\usepackage[numbers]{natbib}
\usepackage{doi}

\newcommand{\mo}{\mathfrak}
\newcommand{\fns}{\footnotesize}
\renewcommand{\iff}{\quad\text{iff}\quad}
\renewcommand{\phi}{\varphi}

\DeclareMathOperator{\incl}{in}

\newcommand{\lan}{\mathcal{L}}
\newcommand{\Star}{\text{\scalebox{.75}{\faStarO}}}
\DeclareMathOperator{\st}{st}
\newcommand{\subs}{\mathrel{\rightsquigarrow}}
\newcommand{\logeq}{\mathrel{\leftrightsquigarrow}}
\newcommand{\Sim}{\text{S}}
\newcommand{\Forw}{F}
\newcommand{\Back}{B}
\newcommand{\FOL}{\mathrm{FOL}}
\newcommand{\Var}{\mathrm{Var}}
\newcommand{\combor}[2]{(#1\mid \mid #2)}
\newcommand{\comband}[2]{(#1\;\& \; #2)}
\newcommand{\classneg}{{\sim}}
\newcommand{\oldwsmile}{\scalebox{1}{$\mathrlap{\Circle}\textcolor{black}{\smallsmile}$}}

\newcommand{\aff}{\mathord{\text{\reflectbox{\rotatebox[origin=c]{90}{$\classneg$}}\kern.3ex}}}
\newcommand{\Simin}[3][]{(#2,#3)\in\Sim_{#1}}
\newcommand{\Simnin}[3][]{(#2,#3)\notin\Sim_{#1}}
\newcommand{\connective}{\heartsuit}

\renewcommand{\smile}{\mathord{\smallsmile}}
\renewcommand{\frown}{\mathord{\smallfrown}}

\newcommand{\cons}{\vartriangleright}
\newcommand{\comp}{\blacktriangleright}

\newcommand{\boxp}{\texttt{[+]}}
\newcommand{\boxm}{\texttt{[-]}}
\newcommand{\diap}{\texttt{<+>}}
\newcommand{\diam}{\texttt{<->}}
\newcommand{\contra}{\ensuremath{{\times}}}%
\newcommand{\dual}{{\ensuremath{{\downharpoonleft}\!{\upharpoonright}}}}%
\newcommand{\ant}{\ensuremath{\rotatebox{90}{$\between$}}}%

\usepackage{graphicx}
\usepackage{mathtools}

\newcommand{\newstackrel}[2]{%
   \mathrel{\vbox{\offinterlineskip\ialign{%
     \hfil##\hfil\cr
     $\scriptscriptstyle#1$\cr
     \noalign{\kern-0.25ex} 
     $#2$\cr
 }}}}

\newcommand{\struc}[1]{\langle#1\rangle}

\usepackage{stackengine}
\newcommand{\wsmile}{\mathord{\textup{\stackon[-1.85pt]{\normalsize${\Circle}$}{\tiny $\smile$}}}}
\newcommand{\wfrown}{\mathord{\textup{\stackunder[-1.65pt]{\normalsize${\Circle}$}{\tiny $\frown$}}}}
\newcommand{\vsmile}{\mathord{\textup{\stackon[-1.85pt]{\normalsize${\CIRCLE}$}{\tiny $\smile$}}}}
\newcommand{\vfrown}{\mathord{\textup{\stackunder[-1.65pt]{\normalsize${\CIRCLE}$}{\tiny $\frown$}}}}
\newcommand{\classconn}{\mathrlap{\wsmile}{\wfrown}}

\newcommand{\simul}{\mathrel{\kern1pt\underline{\kern-.5pt\to\kern-.5pt}\kern1pt}}

\newcommand{\psimul}{\rightleftarrows}
\newcommand{\symsim}{\mathrel{\kern1pt\underline{\kern-.5pt\leftrightarrow\kern-.5pt}\kern1pt}}

\newcommand{\myitem}[1]{
  \renewcommand{\labelenumi}{(\theenumi) }
  \renewcommand{\theenumi}{#1}
  \item
}

\usepackage[noabbrev,capitalise]{cleveref}

\usepackage{amsthm}

  \theoremstyle{definition}
  \swapnumbers
    \newtheorem{para}{}[section]
    \newtheorem{definition}[para]{Definition}
    \newtheorem{examplex}[para]{Example}
    \newtheorem{remark}[para]{Remark}
    
  \theoremstyle{plain}
    \newtheorem{lemma}[para]{Lemma}
    
    \newtheorem{corollary}[para]{Corollary}
    \newtheorem{theorem}[para]{Theorem}
    \newtheorem{proposition}[para]{Proposition}

\newenvironment{example}
  {\pushQED{\qed}\examplex}
  {\popQED\endexamplex}

\title{Intrinsic and relative characterization results \\
for logics with negative modalities}

\author{Jim de Groot$^1$, Jo\~ao Marcos$^2$ and Rodrigo Stefanes$^2$ \\ \medskip
\footnotesize{$^1$ Mathematical Institute, University of Bern, Bern, Switzerland} \\
\footnotesize{$^2$ Universidade Federal de Santa Catarina, Florian\'opolis, Brazil}}

\date{}

\begin{document}

\maketitle

\begin{abstract}
    \noindent
    We introduce simulations for modal logics with subclassical negations and restoration modalities, establish an adequacy theorem, and prove intrinsic (Hennessy--Milner-type) and relative (Van Benthem-type) characterization results. These results identify each restorative language with the fragment of first-order logic preserved by its simulations and delineate the expressive profile of modal logics with non-classical negations.
\end{abstract}

\section{Introduction}

\noindent
Two fundamental results in the model theory of modal logic are associated to the names of Hennessy, Milner, and Van Benthem.
While first-order logic cannot distinguish between isomorphic models, since such models satisfy exactly the same first-order formulas, the \guillemotleft expressive granularity\guillemotright\ of modal logic favors, for the same job, the weaker notion of bisimulation.
In this respect, the Hennessy--Milner theo\-rem provides an \emph{intrinsic characterization} of modal expressivity: 
it establishes a correspondence between bisimilarity, for certain classes of models, and modal equivalence.
The Van Benthem theo\-rem, in turn, offers a \emph{relative characterization} of expressivity: it demonstrates that the modal language is precisely as expressive as the bisimulation-invariant fragment of first-order logic.

In the present work, we arrive at Hennessy--Milner-type and Van Benthem-type results for modal logics and languages endowed with negative modalities and additional restoration operators that enable these modalities to recover the behavior of classical negation. We begin by briefly situating this study within its broader theoretical setting.

\paragraph{Restorative modal logics.} 
We study expansions of positive logic by unary operators: subclassical negations and their restoration companions, all interpreted modally on Kripke frames. Intuitionistic negation provides a familiar example of a negative modality.

Assuming logic to be grounded on the complementary and mutually exclusive judgments of \emph{assertion} and \emph{denial}, one may take the \emph{conjunction} of two formulas to be asserted iff both formulas are asserted; dually, the \emph{disjunction} of two formulas may be taken to be denied iff both formulas are denied; the nullary forms of these connectives, called \emph{top} and \emph{bottom}, are by default asserted and denied, respectively.
\emph{Negation}, in turn, is construed as an operation that inverts judgments: it denies an asserted formula and asserts a denied one. 
A paracomplete negation arises when both a formula and its negation are jointly denied,  
while not denying all other formulas;
a paraconsistent negation arises when both a formula and its negation are jointly asserted, 
while not asserting all other formulas.
In what follows, we focus on logics containing negations with a modal character, for which a certain \emph{global contraposition principle} is expected to hold: whenever one formula follows from another, the negation of the latter follows from the negation of the former. Such behavior often results as a by-product of a \emph{classical-like negation}, namely one that is neither paracomplete nor paraconsistent, and that systematically converts all assertions into denials and vice versa. In the present study, nonetheless, negations will not be automatically assumed to behave classically.%

The modal negations considered here fail to respect certain fundamental classical principles. To \emph{recover} such lost assumptions and \emph{recapture} the full deductive power of classical propositional logic, we introduce a collection of standard \emph{restoration connectives} into the language: 
the consistency of a formula is denied precisely when the formula is asserted together with its (paraconsistent) negation; dually, the undeterminedness of a formula is asserted precisely when the formula is denied together with its (paracomplete) negation. In this framework, notions of [\emph{in}]\emph{consistency} and [\emph{un}]\emph{determinedness} also receive a modal interpretation.
Modal logics enriched with these connectives, namely subclassical negations and their restoration companions, will here be called \emph{restorative modal logics}. As will be shown, they correspond to a highly expressive fragment of classical first-order logic, a fragment that is capable of simulating the behavior of non-classical negations while fully reinstating classical reasoning, despite the fact that, in general, no classical-like negation is therein definable.

\paragraph{Bisimulations.}
  A bisimulation is a relation between worlds of Kripke models that preserves and reflects the truth of propositional letters and mirrors the accessibility relation in both directions.
  As a result, worlds linked by a bisimulation satisfy exactly the same formulas. 
  Also known as \emph{p-relations} and \emph{zig-zag relations},
  bisimulations were first introduced by Van Benthem as part of his work on
  correspondence theory~\cite{Ben76,Ben83,Ben84}.
  They were independently discovered in computer science
  as an equivalence relation between process graphs~\cite{Par81,HenMil85},
  and later employed as a notion of equality in
  non-well-founded set theory~\cite{Acz88,BarMos96,Bal98-PhD}.
  
  From a (modal) logical point of view, bisimulations often serve a double purpose.
  First, they provide a structural notion that guarantees
  logical equivalence. Here by \guillemotleft structural\guillemotright\ we mean that the de\-fi\-ni\-tion
  of a bisimulation only refers to the interpretation of propositional letters
  and to the modal accessibility relation, 
  but not to the interpretation of arbitrary formulas. 
  In special cases, the converse also holds: if two
  worlds satisfy the same formulas then they are linked by some bisimulation.
  This property is called the \emph{Hennessy--Milner property}, and classes of models satisfying it are called \emph{Hennessy--Milner classes},
  after the authors of~\cite{HenMil85}.
  Second, bisimulations establish a precise correspondence between modal logic and first-order logic: under its usual interpretation, the standard modal language corresponds exactly to the bisimulation-invariant fragment of first-order logic with one free variable.
  This result was first proven by Van Benthem~\cite{Ben76}, and is commonly referred to as the \emph{Van Benthem characterization theo\-rem}.
  
  Following Hennessy and Milner's theo\-rem and Van Benthem's theo\-rem for normal modal logic in its customary language with positive modalities, 
  analogous results were obtained for a wide range of logics, including non-normal modal logics~\cite{Han03,HanKupPac09}, modal fixpoint logics~\cite{JanWal95,EnqSeiVen19}, 
  and (bi-)intuitionistic (modal) logic \cite{Pat97,Bad16,GroPat21-hm}.
  Each modal language is equipped with its own appropriate notion of bisimulation.
  Recent work on intuitionistic modal logic has moreover shown that the ambient first-order language need not itself be classical, by characterizing $\mathsf{IK}$ as a bisimulation-invariant fragment of intuitionistic first-order logic~\cite{DeGrootMarcosStefanes2026}.

\paragraph{Simulations.}
  The symmetrical nature of bisimulations entails that they automatically preserve the truth of negations. While this feature is appropriate for modal extensions of classical propositional logic, it is undesirable when seeking characterization results for logics lacking classical-like negations,
  such as positive modal logic~\cite{KurRij97}.
  In such settings, bisimulations are sometimes replaced with weaker relations known as \emph{simulations}. The key difference between bisimulations
  and simulations is that the latter preserve the truth of formulas without necessarily reflecting it.
  
  Similar strategies have been employed, for instance, in the characterization of non-distributive positive
  logic as a fragment of first-order logic~\cite{Gro24-vb-ll}
  and in analyses of fragments of the calculus of binary relations~\cite{FleEA15}.
 Since the restorative modal logics examined in this paper can all be viewed as positive logic augmented with unary modal operators, we likewise make use of simulations.

\paragraph{Contributions.}
  In this paper we develop a notion of simulation tailored to restorative modal logics.
  More precisely, we define a notion of simulation that is parametric in a
  similarity type that describes the intensional operators in the language.
  This poses an interesting challenge: although simulations
  are used to prevent the preservation of the truth of negations, some of the modal operators
  under investigation exhibit an inherently negative character 
  that must, in fact, be preserved.
%
  For the consistency operator, in particular, the resulting simulation condition is closely related to the $\circ$-bisimulation of \cite[Section~5]{Fan15}; Section~\ref{sec:symm} shows that the induced notions of bisimilarity coincide after symmetrization.

  After identifying the correct notion of simulation, in each case, we establish 
  intrinsic (that is, Hennessy--Milner-type) characterization results
  for the class of modally saturated Kripke models.
  This, in turn, gives rise to 
  relative (that is, Van Benthem-type) characterization results 
  for the modal languages that are here under scrutiny. 
  Finally, we outline how analogous results can be obtained for restorative modal logics with a classical-like negation, and we speculate how these findings may
  fit into a general pattern of simulations for \guillemotleft combined modalities\guillemotright.

\paragraph{Structure of the paper.}
  Section~\ref{sec:restoration} supplies the abstract logical background concerning negative modalities, restoration modalities and restorative modal logics. 
  The main model-theoretic development begins next:
  after recalling Kripke semantics for 
  normal modal logics in Subsection~\ref{subsec:classical}, we introduce interpretations
  for negative modalities and for restoration modalities in Section~\ref{subsec:Kripke-restoration}.
  Then, Section~\ref{subsec:negation} discusses, from a modal perspective, when and why classical negation fails to be definable in the present framework.
  In Section~\ref{sec:sim} we introduce simulations for restorative modal logics,
  parametric in their modal similarity types: each modal operator comes with
  its own simulation condition, which is postulated when relevant.
  In Section~\ref{sec:sim-between} we briefly discuss the difference between
  simulations ranging over a single model and simulations between models.
  Next, in Section~\ref{sec:hm} we prove 
  an intrinsic characterization 
  result,
  which is used in Section~\ref{sec:vb} to derive 
  a relative characterization result.
  We then sketch, in Section~\ref{sec:symm}, how to obtain analogous characterization results in the presence of classical-like negations.
  Finally, in Section~\ref{sec:combined} we discuss
  a possible generalization of our work to so-called combined modalities,
  and we wrap things up in Section~\ref{sec:conc}.


\section{Negative modalities and their restoration companions}\label{sec:restoration}

\noindent
We will write $\Pi\comp\Sigma$ to indicate that asserting the formulas in~$\Pi$ is \emph{judgment-compatible} with denying the formulas in~$\Sigma$.  \emph{Consequence} relations may be conceived as \emph{judgment-incompatibility} relations, and accordingly we write $\Pi\cons\Sigma$ (read as \guillemotleft$\Sigma$ follows from $\Pi$\guillemotright) to indicate that $\Pi\comp\Sigma$ fails to be the case.
As usual, a blank antecedent or succedent denotes the empty set of formulas.

\paragraph{Negations of all flavors.} 
Any unary connective~$\neg$ that allows for the assertion of a propositional letter~$p$ to be judgment-compatible with the denial of~$\neg p$ (that is, $p\comp\neg p$), and likewise allows for the denial of~$p$ to be judgment-compatible with the assertion of~$\neg p$ (that is, $\neg p\comp p$) will here deserve the appellation of \emph{negation}~\cite{Mar05}.  A \emph{classical-like negation}~$\sim$ is a negation such that (Neg$_i^\classneg$) $p,{\sim} p\cons q$ and (Neg$_u^\classneg$) $q\cons p,{\sim} p$, for any two propositional letters~$p$ and~$q$.  A~ne\-ga\-tion that fails (Neg$_i^\classneg$) is called \emph{paraconsistent}, whereas a negation that fails (Neg$_u^\classneg$) is called \emph{paracomplete}.

\paragraph{Classical and normal modal logics.} 
The defining feature of \emph{classical modal logics} \cite{Seg82,Che80} is the so-called \emph{replacement property} (also known as \emph{congruentiality}): if two formulas~$\alpha$ and~$\beta$ are \emph{logically equivalent} (that is, $\alpha\cons\beta$ and $\beta\cons\alpha$), then $\varphi[p\mapsto \alpha]$ is logically equivalent to $\varphi[p\mapsto \beta]$, where~$\varphi$ is an arbitrary formula and $\varphi[p\mapsto \gamma]$ indicates that all occurrences of the propositional letter~$p$ in~$\varphi$ have been substituted by the formula~$\gamma$.  
Intuitively, in classical modal logics equivalent formulas are logically indistinguishable, or interchangeable in all contexts. In other words, from an algebraic viewpoint, all \emph{schematic contexts} $\varphi(p)$ are quotient-compatible with the underlying notion of logical equivalence.  If each connective of a given logic ---thought of as inducing, on each of its arguments, a schematic context--- is quotient-compatible with the underlying notion of logical equivalence, the replacement property clearly holds.

Given a unary connective~$\connective$, sufficient conditions for ensuring its quotient-compatibility with logical equivalence consist in, for all formulas~$\alpha$ and~$\beta$, assuming \emph{modal normality} in one of the two following guises~\cite{DodMar14}: 
\smallskip

\begin{tabular}{rl}
  {[prs]} & if $\alpha\cons\beta$, then $\connective\alpha\cons\connective\beta$ \\
  {[rev]} & if $\alpha\cons\beta$, then $\connective\beta\cons\connective\alpha$ \\
\end{tabular}
\smallskip

\noindent
The \guillemotleft preservation\guillemotright\ condition [prs] can be viewed as the monotonicity of~$\connective$ with respect to~$\cons$, and is characteristic of \emph{positive modalities}.
The \guillemotleft reversal\guillemotright\ condition [rev], on the other hand, can be viewed as antitonicity of~$\connective$ with respect to~$\cons$ and is characteristic of \emph{negative modalities}.  In this sense, it is clear that the box and the diamond connectives of usual modal languages are positive modalities, whereas negation in intuitionistic (or classical) logic is a negative modality.

\paragraph{Standard connectives, positive and negative modalities.} 
\emph{Standard} conjunctions, disjunctions, tops and bottoms are characterized by the following principles:
\smallskip

\begin{tabular}{llll}
  {[StdC]} & $\Pi,\varphi_1\land\varphi_2\cons\Sigma$ \iff $\Pi,\varphi_1,\varphi_2\cons\Sigma$
  \hspace{2cm} &
  {[StdT]} & $\Pi,\top\cons\Sigma$ \iff $\Pi\cons\Sigma$\\
  {[StdD]} & $\Pi\cons\varphi_1\lor\varphi_2,\Sigma$ \iff $\Pi\cons\varphi_1,\varphi_2,\Sigma$ &
  {[StdB]} & $\Pi\cons\bot,\Sigma$ \iff $\Pi\cons\Sigma$\\
\end{tabular}
\smallskip

\noindent 
for an arbitrary collection of formulas $\Pi\cup\Sigma\cup\{\varphi_1,\varphi_2\}$.
In the following, consider a logic containing all the above-mentioned standard connectives, let~$+$ denote an arbitrary unary positive modality, and let~$-$ denote an arbitrary unary negative modality (that is, assume $+$ respects [prs] and assume $-$ respects [rev]).  In that case, it is easy to check the following, for arbitrary propositional letters $p$ and $q$ \cite{DodMar14}: 
\smallskip

\begin{tabular}{llll}
    {[PM1.1]} & 
    $+(p\land q)\cons+p\land+q$ 
    \hspace{2cm} 
    & {[PM2.1]} & 
    $+p\lor+q\cons+(p\lor q)$ 
    \\
    {[DM1.1]} & 
    $-(p\lor q)\cons-p\land-q$ 
    \hspace{2cm} 
    & {[DM2.1]} & 
    $-p\lor-q\cons-(p\land q)$ 
    \\
\end{tabular}
\smallskip

The converses of the latter conditions characterize the behavior of \guillemotleft boxes\guillemotright\ and \guillemotleft diamonds\guillemotright\ in the presence of the above standard connectives.  Specifically: 
\smallskip

\begin{tabular}{llll}
    & \underline{\smash{\emph{box-plus}}}: \boxp & & \underline{\smash{\emph{diamond-plus}}}: \diap \\
    {[PM1.2]} & 
    $+p\land+q\cons+(p\land q)$ 
    \hspace{2cm} 
    & {[PM2.2]} & 
    $+(p\lor q)\cons+p\lor+q$ 
    \\
    {[PM1.0]} & $\cons+\top$ & {[PM2.0]} & $+\bot\cons$\\[1mm]
    & \underline{\smash{\emph{box-minus}}}: \boxm & & \underline{\smash{\emph{diamond-minus}}}: \diam \\
    {[DM1.2]} & 
    $-p\land-q\cons-(p\lor q)$ 
    \hspace{2cm} 
    & {[DM2.2]} & 
    $-(p\land q)\cons-p\lor-q$ \\
    {[DM1.0]} & $\cons-\bot$ & {[DM2.0]} & $-\top\cons$\\
\end{tabular}
\smallskip

\noindent 
Well-known examples of box-plus and diamond-plus connectives are given, respectively, by the positive modalities~$\Box$ and~$\Diamond$ in normal modal logics.
Classical-like negations verify all of the [DM$\cdots$] (\guillemotleft De Morgan\guillemotright) conditions.
Subclassical box-minus connectives constitute typical examples of paracomplete negations (of which, again, intuitionistic negation constitutes a canonical illustration), whereas subclassical diamond-minus connectives constitute typical examples of paraconsistent negations.

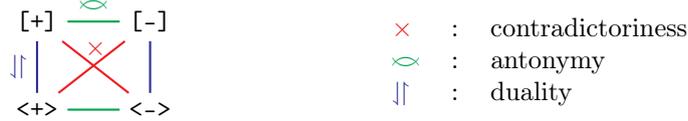
\begin{figure}[h!]
  \begin{minipage}{0.45\textwidth}
  \centering
\begin{tikzpicture}
  \matrix [matrix of math nodes,row sep=.7cm,column sep=.7cm,ampersand replacement=\&]
  {
      |(A)| \boxp
   \& |(E)| \boxm \\ 
      |(I)| \diap
   \& |(O)| \diam \\ 
  };
  \begin{scope}[every node/.style={auto}]
     {\draw[-,Red,thick]   (A) to node[above] {{$\contra$}} (O);}
     {\draw[-,Red,thick]   (E) to node        {}            (I);}
     {\draw[-,Blue,thick]  (A) to node[left]  {{$\dual$}}   (I);}
     {\draw[-,Blue,thick]  (E) to node        {}            (O);}
     {\draw[-,Green,thick] (A) to node        {{$\ant$}}    (E);}
     {\draw[-,Green,thick] (I) to node        {}            (O);}
  \end{scope}
\end{tikzpicture}
\end{minipage}
\begin{minipage}{0.45\textwidth}
\begin{tabular}{lcl}
$\textcolor{Red}{\contra}$ & : & contradictoriness \\
$\textcolor{Green}{\ant}$ & : & antonymy \\
$\textcolor{Blue}{\dual}$ & : & duality \\
\end{tabular}
\end{minipage}
    \caption{Square of Oppositions involving positive and negative modalities.}
    \label{fig:modalities}
\end{figure}

\paragraph{Restoration connectives.} 
When a classical-like negation is replaced by a subclassical one, some assumption or principle characteristic of classical negation is inevitably absent.  A richer language may often be devised, though, in order to allow for such lost assumptions to be explicitly recovered.  This is the very idea behind the introduction of the so-called \emph{restoration connectives}: here, their role is to allow for the properties of classical negation to be recaptured.

Recall that a paraconsistent negation~$\smile$ is a negation that allows, by rejecting (Neg$_i^{\smile}$), for the joint assertion of a formula~$\varphi$ and its negation~$\smile\varphi$, while still denying some other formula.  This is a sign of inconsistency.  One may entertain a language that contains a connective~$\vsmile$ that tracks down \emph{$\smile$-inconsistency}, in such a way that the formula $\vsmile p$ is asserted exactly when~$p$ and $\smile p$ are at once asserted; meanwhile, \emph{$\smile$-consistency} could be tracked down by a connective~$\wsmile$ in such a way that the formula $\wsmile p$ is asserted exactly when at least one among~$p$ and $\smile p$ is denied (see [Con$\cdots$], below).  
A \emph{basic} legitimate consistency companion for the paraconsistent negation~$\smile$ would actually be a restoration connective~$\wsmile$ that allows for any two formulas among $\varphi$, $\smile \varphi$ and $\wsmile \varphi$ to be simultaneously asserted ([LegCa] and [LegCb]), yet not for all of the three of them to be asserted at once ([Con1]). 
We note that \emph{Logics of Formal Inconsistency} were originally defined  \cite{carnielli:marcos:taxonomy,carnielli:coniglio:marcos:HPL} precisely as $\wsmile$-gently explosive ([Con1]) $\smile$-paraconsistent logics respecting the legitimacy conditions [LegCa] and [LegCb].  The intuition is that $\wsmile$ would be used to explicitly signal a \guillemotleft consistency assumption\guillemotright\ concerning the behavior of~$\smile$.  
If all negated formulas satisfy this assumption, the \guillemotleft classicality\guillemotright\ principle 
$\cons\wsmile p$ holds universally, making the presence of the connective $\wsmile$ in the language superfluous.
The \emph{standard} form of \emph{gentle explosion} is supposed to respect additional conditions, [Con2a] and [Con2b], that act as converse to [Con1], as they guarantee that denying either a formula or its negation implies asserting the consistent behavior of this formula with respect to negation.  
Thus, for standard gently explosive logics, asserting the negation-inconsistency of a formula is both a necessary condition (by \guillemotleft gentleness\guillemotright) and a sufficient condition (by \guillemotleft standardness\guillemotright) for the formula and its negation to be jointly asserted.
Not all Logics of Formal Inconsistency contain standard restoration connectives.%

Here is a summary of the above-mentioned principles:

\begin{center}
\begin{tabular}{llllllll}
 [ConInc] & $\Pi,\vsmile\varphi\cons\Sigma$ \iff $\Pi\cons\wsmile\varphi,\Sigma$ \hspace{1cm} & [LegCa] & $p,\wsmile p\comp$ & [LegCb] & $\smile p,\wsmile p\comp$ \\
 {[Con1]} & $\wsmile p,p,\smile p\cons$ \hspace{1cm} & [Con2a] & $\cons \wsmile p, p$ \hspace{1cm} & [Con2b] & $\cons \wsmile p, \smile p$ \\ 
\end{tabular}
\end{center}
Whenever they are present, the connectives $\wsmile$ and $\vsmile$ are dubbed \emph{restoration companions} of~$\smile$.

Dually, by interchanging assertions and denials, one obtains 
$\frown$-paracomplete logics, defining classes of so-called \emph{Logics of Formal Undeterminedness} \cite{Mar:paranormal}. The latter contain restoration connectives~$\wfrown$ and~$\vfrown$ that track down $\frown$-determinedness and $\frown$-indeterminedness, in such a way that the formula $\vfrown p$ is asserted, and the formula $\wfrown p$ is denied, exactly when $p$ and $\frown p$ are at once denied.
Accordingly, here are the relevant principles that characterize the \emph{standard} form of \emph{gentle implosion}:

\begin{center}
\begin{tabular}{llllllll}
 [UndDet] & $\Pi,\wfrown\varphi\cons\Sigma$ \iff $\Pi\cons\vfrown\varphi,\Sigma$ \hspace{1cm} & [LegDa] & $\comp p,\vfrown p$ & [LegDb] & $\comp\frown p,\vfrown p$ \\
 {[Und1]} & $\cons p,\frown p,\vfrown p$ \hspace{1cm} & [Und2a] & $p, \vfrown p\cons$ \hspace{1cm} & [Und2b] & $\frown p, \vfrown p\cons$ \\ 
\end{tabular}
\end{center}
Whenever present, the connectives $\vfrown$ and $\wfrown$ are referred to as the \emph{restoration companions} of~$\frown$.
The underlying recapturing idea, in both cases, is that (Neg$_i^{\smile}$) is recovered precisely when the \guillemotleft consistency assumption\guillemotright\ is explicitly asserted, while (Neg$_u^{\frown}$) is recovered precisely when the \guillemotleft determinedness assumption\guillemotright\ is explicitly asserted.

\begin{figure}[htbp]
\centering
  \begin{minipage}{0.3\textwidth}
  \centering
{%
\begin{tikzpicture}\hspace{0mm}
  \matrix [matrix of math nodes,row sep=.7cm,column sep=.7cm,ampersand replacement=\&]
  {
      |(A)| \wsmile p, p\comp %
   \& |(E)| {\wfrown} p\comp p \\ %
      |(I)| \comp p,{\vfrown} p  %
   \& |(O)| p\comp \vsmile p \\ %
  };
  \begin{scope}[every node/.style={auto}]
     {\draw[-,Red,thick]   (A) to node[above] {{$\contra$}} (O);}
     {\draw[-,Red,thick]   (E) to node        {}            (I);}
     {\draw[-,Blue,thick]  (A) to node[left]  {{$\dual$}}   (I);}
     {\draw[-,Blue,thick]  (E) to node        {}            (O);}
     {\draw[-,Green,thick] (A) to node        {{$\ant$}}    (E);}
     {\draw[-,Green,thick] (I) to node        {}            (O);}
  \end{scope}
\end{tikzpicture}%
}%
\end{minipage}
\begin{minipage}{0.3\textwidth}
  \centering
{%
\begin{tikzpicture}\hspace{-2mm}
  \matrix [matrix of math nodes,row sep=.7cm,column sep=.7cm,ampersand replacement=\&]
  {
      |(A)| \wsmile p, \smile p\comp %
   \& |(E)| {\wfrown} p\comp {\frown} p \\ %
      |(I)| \comp {\frown} p,{\vfrown} p  %
   \& |(O)| \smile p\comp \vsmile p \\ %
  };
  \begin{scope}[every node/.style={auto}]
     {\draw[-,Red,thick]   (A) to node[above] {{$\contra$}} (O);}
     {\draw[-,Red,thick]   (E) to node        {}            (I);}
     {\draw[-,Blue,thick]  (A) to node[left]  {{$\dual$}}   (I);}
     {\draw[-,Blue,thick]  (E) to node        {}            (O);}
     {\draw[-,Green,thick] (A) to node        {{$\ant$}}    (E);}
     {\draw[-,Green,thick] (I) to node        {}            (O);}
  \end{scope}
\end{tikzpicture}%
}%
\end{minipage}\smallskip\\
  \begin{minipage}{0.37\textwidth}
  \centering
\begin{tikzpicture}\hspace{-2mm}
  \matrix [matrix of math nodes,row sep=.7cm,column sep=.7cm,ampersand replacement=\&]
  {
      |(A)| \wsmile p,p,\smile p\cons %
   \& |(E)| {\wfrown} p\cons p,{\frown} p \\ %
      |(I)| \cons p,{\frown} p,{\vfrown} p %
   \& |(O)| p,\smile p\cons\vsmile p \\ %
  };
  \begin{scope}[every node/.style={auto}]
     {\draw[-,Red,thick]   (A) to node[above] {{$\contra$}} (O);}
     {\draw[-,Red,thick]   (E) to node        {}            (I);}
     {\draw[-,Blue,thick]  (A) to node[left]  {{$\dual$}}   (I);}
     {\draw[-,Blue,thick]  (E) to node        {}            (O);}
     {\draw[-,Green,thick] (A) to node        {{$\ant$}}    (E);}
     {\draw[-,Green,thick] (I) to node        {}            (O);}
  \end{scope}
\end{tikzpicture}
  \end{minipage}
  \begin{minipage}{0.37\textwidth}
  \centering
\begin{tikzpicture}\hspace{-2mm}
  \matrix [matrix of math nodes,row sep=.7cm,column sep=.7cm,ampersand replacement=\&]
  {
      |(A)| 
        \begin{tabular}{l}
        $\cons\wsmile p,p$\\
        $\cons\wsmile p,\smile p$
        \end{tabular}%
   \& |(E)| 
       \begin{tabular}{l}
        $p\cons\wfrown p$\\
        $\smile p\cons\wfrown p$
        \end{tabular}%
    \\ %
      |(I)| 
       \begin{tabular}{l}
        $\vfrown p,p\cons$\\
        $\vfrown p,\frown p\cons$
        \end{tabular}%
   \& |(O)| 
       \begin{tabular}{l}
        $\vsmile p\cons p$\\
        $\vsmile p\cons \smile p$
        \end{tabular}%
    \\ %
  };
  \begin{scope}[every node/.style={auto}]
     {\draw[-,Red,thick]   (A) to node[above] {{$\contra$}} (O);}
     {\draw[-,Red,thick]   (E) to node        {}            (I);}
     {\draw[-,Blue,thick]  (A) to node[left]  {{$\dual$}}   (I);}
     {\draw[-,Blue,thick]  (E) to node        {}            (O);}
     {\draw[-,Green,thick] (A) to node        {{$\ant$}}    (E);}
     {\draw[-,Green,thick] (I) to node        {}            (O);}
  \end{scope}
\end{tikzpicture}
  \end{minipage}
    \caption{Squares of Oppositions containing the restoration companions of the subclassical negations. The top squares represent the \emph{legitimacy} conditions; at the bottom, the left square contains the \emph{basic} restorative conditions and the right square supplements them with the \emph{standard} restorative conditions.}
    \label{fig:restoration}
\end{figure}
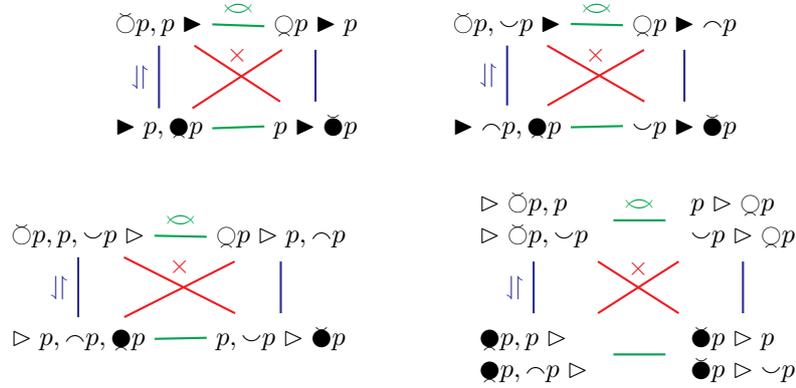

The Squares of Oppositions represented in Figures~\ref{fig:modalities} and~\ref{fig:restoration} fix and extend the ones found in~\cite{Mar:paranormal}.

\section{Kripke semantics for modal logics}

\noindent 
As we shall entertain a range of unary operators, we start by introducing notation for the languages thereby induced:

\begin{definition}\label{def:languages}
  By a \emph{similarity type} we shall refer to a collection of unary connectives. 
  Fix some arbitrary non-empty set~$K$ and let $\Lambda$ be a similarity type.
  We define $\lan_{\Lambda}$ to be the language generated by the grammar
  \begin{equation*}
    \phi ::= p_k
      \mid \top
      \mid \bot
      \mid \phi \wedge \phi
      \mid \phi \vee \phi
      \mid \Star\phi
  \end{equation*}
  where $k$ ranges over $K$ and $\Star \in \Lambda$.
  We abbreviate $\lan_{\Star_1, \ldots, \Star_n} := \lan_{\{\Star_1, \ldots, \Star_n\}}$.
\end{definition}

\subsection{Modal logic as you know it}\label{subsec:classical}

\noindent
  We briefly recall some basic de\-fi\-ni\-tions regarding
  relational semantics for modal languages.
  Formulas in the well-known modal language $\lan_{\classneg,\Box,\Diamond}$ will be interpreted as follows:

\begin{definition}\label{def:box}\label{def:KripkeModel}
A \emph{Kripke model}~$\mo{M}$ is a structure $\struc{W, R, \{P_k\}_{k\in K}}$ consisting of a nonempty set
  $W$ of \guillemotleft worlds\guillemotright, a binary \guillemotleft accessibility relation\guillemotright\ $R$ on $W$, and a unary predicate~$P_k$ on~$W$ for each $k \in K$. 
  Each such~$P_k$ represents the \guillemotleft proposition expressed by the propositional letter~$p_k$ within~$\mo{M}$\guillemotright.
  We refer to~$P_k$ as the \emph{truth set of~$p_k$}.
  Given $w \in W$, we let $R[w] := \{v \in W \mid w R v\}$ be the set of worlds accessible from~$w$.
  The satisfaction of a formula $\varphi \in\lan_{\classneg,\Box,\Diamond}$ at a world $w$ of the Kripke model $\mo{M}$, denoted by $\mo{M}, w \Vdash \varphi$, is recursively defined as follows:
  \begin{enumerate}
 \setlength{\itemindent}{2em}
    \myitem{Sat$_k$}
    $\mo{M}, w \Vdash p_k \iff w\in P_k$
    \myitem{Sat$\top$}
    $\mo{M}, w \Vdash \top $ always holds
    \myitem{Sat$\bot$}
    $\mo{M}, w \Vdash \bot $ never holds
    \myitem{Sat$\land$}
    $\mo{M}, w \Vdash \varphi_1 \land \varphi_2 \iff \mo{M}, w \Vdash \varphi_1 \text{ and } \mo{M}, w \Vdash \varphi_2$ 
    \myitem{Sat$\lor$}
    $\mo{M}, w \Vdash \varphi_1 \lor \varphi_2 \iff \mo{M}, w \Vdash \varphi_1 \text{ or } \mo{M}, w \Vdash \varphi_2$
    \myitem{Sat$\classneg$}\label{it:sat:classneg}
          $\mo{M}, w \Vdash \classneg\phi \iff \mo{M}, w \not\Vdash \phi$
    \myitem{Sat$\Box$}\label{it:sat:box} 
          $\mo{M}, w \Vdash \Box\phi \iff \text{for all } v \in W,
            \text{ if } w R v
            \text{ then } \mo{M}, v \Vdash \phi$
    \myitem{Sat$\Diamond$}\label{it:sat:diamond}
          $\mo{M}, w \Vdash \Diamond\phi \iff \text{there exists } v \in W
            \text{ such that } w R v
            \text{ and } \mo{M}, v \Vdash \phi$
  \end{enumerate}
\end{definition}

\noindent
Clearly, $\Box$ and $\Diamond$ are interdefinable with the help of $\classneg$.  This interdefinability ceases if~$\classneg$ is omitted from the language.

  We note that the negation symbol~$\classneg$ is interpreted \guillemotleft classically\guillemotright\ (at each given world).  
In fact, using the terminology and de\-fi\-ni\-tions introduced in Section~\ref{sec:restoration}, and identifying assertion (resp.\ denial) with satisfaction (resp.\ non-satisfaction), one may now easily check that: 
(0) $\land$, $\lor$, $\top$, $\bot$ respect all the corresponding standard principles;
(1) $\Diamond$ is a diamond-plus connective; 
(2) $\Box$ is a box-plus connective; 
(3) $\sim$ is a classical-like negation.

\begin{remark}
  Kripke models are customarily presented in the literature as structures of the form $\struc{\struc{W,R},V}$, where the directed graph $\struc{W,R}$ is a \guillemotleft Kripke frame\guillemotright\ and the \guillemotleft valuation\guillemotright\ mapping $V:\{p_k\mid k\in K\}\to 2^W$ assigns to each propositional letter the set of all worlds in which it is true, that is, for each $k\in K$ we have that $V(p_k)=\{w\in W\mid \mo{M},w\Vdash p_k\}=P_k$.  However, for our present purposes we find it more perspicuous to present Kripke models simply as first-order structures interpreting a signature containing one binary relation symbol, $\mathbf{R}$, and one unary predicate symbol, $\mathbf{P}_k$, for each $k \in K$.
  For that matter, see also De\-fi\-ni\-tion~\ref{fol:sign}.
\end{remark}

  We will sometimes merge two Kripke models as follows:

\begin{definition}\label{def:disjunion}
  Let $\mo{M}_1 = \struc{W_1, R_1, \{P_{k,1}\}_{k\in K}}$ and $\mo{M}_2 = \struc{W_2, R_2, \{P_{k,2}\}_{k\in K}}$ be two
  Kripke models. Let:%
  \begin{align*}
    W &:= \{ (\ell, w) \mid \ell \in \{ 1, 2 \} \text{ and } w \in W_\ell \} \\
    R &:= \{ ((\ell, w), (\ell, v)) \mid \ell \in \{ 1, 2 \} \text{ and } (w, v) \in R_\ell \} \\
    P_k &:= \{ (\ell, w) \mid \ell \in \{ 1, 2 \} \text{ and } w \in P_{k,
    \ell} \}\text{, for each }k\in K 
  \end{align*}
  Then the \emph{disjoint union} of $\mo{M}_1$ and $\mo{M}_2$ is the Kripke
  model $\mo{M}_1 \uplus \mo{M}_2 := \struc{W, R,  \{P_{k}\}_{k\in K}}$.
\end{definition}

  The following result confirms that the two Kripke models given as input in the de\-fi\-ni\-tion above live inside their disjoint union.  It may be proven using a routine inductive argument, or using the fact that the inclusion maps are bounded morphisms:%

\begin{proposition}
  Let $\mo{M}_1 = \struc{W_1, R_1,\{P_{k,1}\}_{k\in K}}$ and $\mo{M}_2 = \struc{W_2, R_2,\{P_{k,2}\}_{k\in K}}$ be two
  Kripke models, and for each $\ell \in \{ 1, 2 \}$ let
  $\incl_\ell : \mo{M}_\ell \to \mo{M}_1 \uplus \mo{M}_2$
  be the inclusion map given by $w \mapsto (\ell, w)$.
  Then for all $w \in W_\ell$ and $\phi \in \lan_{\classneg,\Box,\Diamond}$,
  \begin{equation*}
    \mo{M}_\ell, w \Vdash \phi \iff (\mo{M}_1 \uplus \mo{M}_2), \incl_\ell(w) \Vdash \phi.
  \end{equation*}
\end{proposition}

\subsection{Interpreting negative modalities and their restoration companions}\label{subsec:Kripke-restoration}

\noindent
In this paper, by \guillemotleft positive logic\guillemotright\ we mean the logic of top, bottom, conjunction and disjunction, all interpreted as in classical logic.
We introduce below the unary connectives~$\smile$ and~$\frown$ in order to represent \guillemotleft subclassical negations\guillemotright, and the unary connectives~$\wsmile$, $\wfrown$, $\vsmile$, and $\vfrown$ to represent \guillemotleft restoration connectives\guillemotright.

\begin{definition}\label{def:simil}
  A \emph{restorative similarity type} is a subset of 
  $\{ \smile, \frown, \wsmile, \wfrown, \vsmile, \vfrown \}$.
\end{definition}

  Restorative modal logics will be interpreted in Kripke models as follows:

\begin{definition}\label{sat:simil}
  Let $\mo{M} = \struc{W, R, \{P_{k}\}_{k\in K}}$ be a Kripke model
  and let~$\Lambda$ be a restorative similarity type.
  Then the in\-ter\-pre\-tation of a formula $\phi \in \lan_{\Lambda}$
  at a world $w \in W$ is recursively defined via the relevant
  clauses from De\-fi\-ni\-tion~\ref{def:box} plus the following clauses: 
  \smallskip

\begin{enumerate}
 \setlength{\itemindent}{2em}
    \myitem{Sat$\smile$}\label{it:sat:smile}   
          $\mo{M}, w \Vdash \smile\phi \iff \text{there exists } v \in W
            \text{ such that } wRv
            \text{ and } \mo{M}, v \not\Vdash \phi$
    \myitem{Sat$\frown$}\label{it:sat:frown}  
          $\mo{M}, w \Vdash \frown\phi \iff \text{for all } v \in W
            \text{ such that } wRv
            \text{ we have } \mo{M}, v \not\Vdash \phi$
    \myitem{Sat$\wsmile$}\label{it:sat:wsmile}  
          $\mo{M}, w \Vdash \wsmile\phi \iff \mo{M}, w \not\Vdash \phi
            \text{ or }(\text{for all } v \in W
            \text{ such that } wRv
            \text{ we have } \mo{M}, v \Vdash \phi)$
    \myitem{Sat$\wfrown$}\label{it:sat:wfrown}  
          $\mo{M}, w \Vdash \wfrown\phi \iff \mo{M}, w \Vdash \phi
            \text{ or }(\text{for all } v \in W
            \text{ such that } wRv
            \text{ we have } \mo{M}, v \not\Vdash \phi)$
    \myitem{Sat$\vsmile$}\label{it:sat:vsmile}  
          $\mo{M}, w \Vdash \vsmile\phi \iff \mo{M}, w \Vdash \phi
            \text{ and }(\text{there exists } v \in W
            \text{ such that } wRv
            \text{ and } \mo{M}, v \not\Vdash \phi)$ 
    \myitem{Sat$\vfrown$}\label{it:sat:vfrown}  
          $\mo{M}, w \Vdash \vfrown\phi \iff \mo{M}, w \not\Vdash \phi
            \text{ and }(\text{there exists } v \in W
            \text{ such that } wRv
            \text{ and } \mo{M}, v \Vdash \phi)$
\end{enumerate}

  Let $w_1$ and $w_2$ be two worlds in two Kripke models
  $\mo{M}_1$ and $\mo{M}_2$, respectively.
  Then we write $\mo{M}_1, w_1 \subs_\Lambda \mo{M}_2, w_2$
  if
  $\mo{M}_1, w_1 \Vdash \phi$ implies $\mo{M}_2, w_2 \Vdash \phi$, for every $\phi \in \lan_{\Lambda}$.
  If the Kripke models are clear from the context, we simply write
  $w_1 \subs_{\Lambda} w_2$.
  We will refer to~$\subs_{\Lambda}$ as the \emph{subsumption relation},
  and say that $w_2$ \emph{subsumes} $w_1$.
\end{definition}

%
  Crucially, all the above modalities may be viewed as abbreviations in $\lan_{\classneg,\Box,\Diamond}$.
  For example, $\smile\phi$ corresponds to $\Diamond\classneg\phi$
  (a.k.a.\ \guillemotleft negation as nonnecessity\guillemotright)
  and $\frown\phi$ corresponds to $\Box\classneg\phi$ (a.k.a.\ \guillemotleft negation as impossibility\guillemotright).
  The restoration modalities then arise from conjoining or disjoining
  these with either $\phi$ or $\classneg\phi$.
  Besides, some restoration modalities can be defined using $\smile$ or $\frown$.
  For example, $\vsmile\phi$ clearly corresponds to $\phi \wedge \smile\phi$.
  However, in the absence of a classical negation, most of these interdefinability features vanish.

Using the terminology and de\-fi\-ni\-tions introduced in Section~\ref{sec:restoration}, one may now easily check that: (1) $\smile$ is a diamond-minus (paraconsistent) negation; (2) $\frown$ is a box-minus (paracomplete) negation; 
(3) $\wsmile$, $\vsmile$, $\wfrown$ and~$\vfrown$ are all standard restoration connectives.
In addition, the semantic clauses in De\-fi\-ni\-tion~\ref{sat:simil} show that all negations and restoration connectives studied here have an intensional character.
  
  We discuss next some logics in the literature that arise from choosing some specific restorative similarity types, interpreted according to the clauses in De\-fi\-ni\-tion~\ref{sat:simil}.

\begin{example}\label{ex:impossibility}
  If we take $\Lambda = \{ \frown \}$, then we obtain positive logic with an \emph{impossibility} operator with the flavor of an
  \emph{intuitionistic-like negation}.
  In the presence of an intuitionistic implication, $\to$, this modality is axiomatizable by
  \begin{equation*}
    \dfrac{\vdash\phi \to \psi}{\vdash\frown\psi \to \frown\phi},
    \qquad
    \vdash(\frown\phi \wedge \frown\psi) \to \frown(\phi \vee \psi)
    \qquad\text{and}\qquad
    \vdash\frown\bot.
  \end{equation*}
  Note that, except for the use of implication, this axiomatization consists basically in guaranteeing that~$\frown$ is a box-minus negative modality (recall [rev], [DM1.2] and [DM1.0], from Section~2).
  This modality was used in~\cite{Dos84a,Dos86} to produce an expansion of intuitionistic propositional logic,
  and was also studied as being added to positive logic without implication~\cite{Cel99,DunZho05,Cel13,LinMa20}.
\end{example}

\begin{example}\label{ex:accident}
  Taking $\Lambda = \{ \wsmile, \vsmile \}\cup\{\classneg\}$, a non-normal modal extension of classical propositional logic is obtained (see \cite{Mar04:accident,steinsvold:2008,kushida:2010,GilVer20}). 
  In this logic, $\vsmile$ is intended to represent an \guillemotleft accidentality operator\guillemotright, in the following assertoric sense: $\vsmile p$ is true at a given world if $p$ is true there, but possibly false (that is, false at an accessible world).
\end{example}

\begin{example}\label{ex:LFIs}
Fixing $\Lambda = \{ \smile, \wsmile, \vsmile \}$, modal Logics of Formal Inconsistency were first introduced in \cite[Chapter~3]{marcos:2005:phd}, and have also been presented using a full restorative language \cite{Mar:paranormal}.
The consistency connective has often been represented as a small white circle, $\circ$, but also on many occasions as a smiling circle, $\oldwsmile$. Do note that the $\frown$-determinedness connective has been variously represented as a star, or ---somewhat confusingly--- as a black frowning circle, rather than the symbol $\wfrown$ that we chose to use here.  

Given that Kripke semantics is involved, it makes sense to consider the effect of selecting specific classes of Kripke frames.  
The logics corresponding to these classes have been examined in various settings \cite{DodMar14,LahMarZoh16,lahav_ori_sequent_2017}, and various other studies explore related systems while not explicitly isolating the restoration connectives that are definable within them \cite{Beziau:Z:2006,Nasiniewsk2:2008,Avron:Zamansky:self:2020}.  
It turns out, in particular, that the class of reflexive frames is axiomatized by a form of excluded middle (namely, $\vdash\varphi\lor\smile\varphi$), and with the help of a standard implication the class of symmetric frames is axiomatized by a form of double-negation elimination (namely, $\vdash\smile\smile\varphi\to\varphi$).
Analytic sequent systems for (implication-free fragments of) such systems have been studied in \cite{LahMarZoh16,lahav_ori_sequent_2017}, and other well-behaved proof systems for some of these systems have also been investigated \cite{Nalon:Marcos:Dixon:2014,Samadpour-et-al:confluence:2025}.
\end{example}

\begin{example}\label{ex:rich-neg}
  Logics defined over the language obtained by taking $\Lambda = \{ \smile, \frown, \wsmile, \vfrown \}$ were
  studied in \cite{Mar:paranormal,DodMar14,LahMarZoh16,lahav_ori_sequent_2017}.
  In this setting, $\top$ and $\bot$ can be viewed as derived operators.
  In certain situations, $\wsmile$ happens to be at once dual and contradictory (recall Figure~\ref{fig:restoration}) to $\vfrown$.  This means, in particular, that the connectives $\wsmile$ and $\wfrown$ collapse into one another, and one may thus refer to them as a \guillemotleft classicality\guillemotright\ connective, $\classconn$, for such a connective allows for both consistent and determined behavior of negation to be recovered at once.  \guillemotleft Perfection operators\guillemotright\ consisting in algebraic counterparts of such a connective were studied in \cite{Esteva:Figallo:Flaminio:Godo:2021,Gomes:Greati:Marcelino:Marcos:Rivieccio:2022}.
\end{example}

\subsection{On the (un)definability of classical negation}\label{subsec:negation}

\noindent
In this section we will focus on restorative modal logics whose semantics satisfy all clauses mentioned in De\-fi\-ni\-tion~\ref{sat:simil}.  As usual within the context of Kripke semantics, we will dub \emph{classical negation} a connective interpreted according to condition \eqref{it:sat:classneg} from De\-fi\-ni\-tion~\ref{def:box}.  Several results concerning the undefinability of a classical negation were established in \cite{lahav_ori_sequent_2017}. Some of these results are recalled in what follows.%

\begin{proposition}\label{prop:undef-classneg}
    A classical negation is not definable in the minimal restorative normal modal logic $PK_\Lambda$, determined by the class of all Kripke frames and based on $\Lambda=\{\smile,\frown,\wsmile,\vfrown\}$.
\end{proposition}

The proof of the above result, given in \cite[Theo\-rem~6.1]{lahav_ori_sequent_2017}, is quite involved, using a sophisticated semantical argument and an induction over all unary formulas that could be considered as candidates for playing the role of a classical negation.  In contrast, by making use of the tools we introduce in the present paper, a much simpler simulation-based proof of the corresponding undefinability claim is given.
In fact, in Example~\ref{exm:neg-new} we show that classical negation is undefinable in $\lan_{\Lambda}$ for any restorative similarity type $\Lambda$.

\begin{remark}
As mentioned in Example~\ref{ex:LFIs}, deductively stronger restorative modal logics have also been studied, by restricting attention to specific classes of frames.  In contrast to the above result, one may show that a classical negation is indeed available at some of these logics.  Over the logics $PKT_{\Lambda}$ determined by the class of reflexive frames, for instance, the classical negation of~$p$ may be explicitly defined by considering $\smile p\land\wsmile p$ or by considering $\frown p\lor\vfrown p$, according to which restorative similarity type~$\Lambda$ one happens to entertain.  Moreover, over the weaker logics $PKD_{\Lambda}$ determined by the class of serial frames one may consider, for instance, $(\smile p\lor\vfrown p)\land\wsmile p$ or $(\frown p\land\wsmile p)\lor\vfrown p$.
\end{remark}

Undefinability results are also known for stronger restorative modal logics (again, see \cite[Theo\-rem~6.1]{lahav_ori_sequent_2017}): 

\begin{proposition}
    A classical negation is also not definable in the restorative modal logics $PKB_\Lambda$ and $PK4_\Lambda$, respectively determined by the class of symmetric frames and by the class of transitive frames, based on $\Lambda=\{\smile,\frown,\wsmile,\vfrown\}$.
\end{proposition}

Note that as soon as positive logic is upgraded through the presence of a classical negation~$\classneg$, then all classical connectives turn out to be definable as usual, and if negative modalities are available then the positive modalities are also definable: indeed, $\Box \varphi$ corresponds to $\classneg\smile\varphi$ (or to $\frown\classneg\varphi$), and $\Diamond \varphi$ corresponds to $\smile\classneg\varphi$ (or to $\classneg\frown\varphi$).  As the negative modalities from $\lan_{\classneg,\smile,\frown}$ are obviously definable in $\lan_{\classneg,\Box,\Diamond}$, these are essentially the same logics.
This explains our focus on logics with negative (and restorative) modalities in which classical negation is \textit{not} definable.  Analogously, in the presence of a connective~$\to$ that behaves as classical implication, it is easy to define a classical negation with the help of~$\bot$.  We will, thus, be primarily interested here in logics \textit{without} a classical implication.

\begin{remark}
In Example~\ref{ex:impossibility}, a language containing a box-minus connective and an intuitionistic implication is mentioned.  The problem of adding an intuitionistic-like implication to a restorative modal logic containing a classicality connective (see Example~\ref{ex:rich-neg}) is a rather non-trivial one.  On this topic, the interested reader is invited to consult \cite{Greati:Marcelino:Marcos:Rivieccio:2024}.
\end{remark}

\section{Simulations}\label{sec:sim}

\noindent
  With the aim of obtaining characterization results for restorative modal logics, we introduce the notion of simulation.
  The de\-fi\-ni\-tion is given parametrically in a restorative similarity type~$\Lambda$, and specifies a relation ranging over a single Kripke model. Simulations between Kripke models will be defined as simulations on the disjoint union of the models.

\begin{definition}\label{def:sim}
  Let $\mo{M} = \struc{W, R, \{P_{k}\}_{k\in K}}$ be a Kripke model.
  A relation $\Sim \subseteq W \times W$ is a
  \emph{$\Lambda$-simulation} if it satisfies~\eqref{it:sim-prop} for every $k \in K$,
  and (Sim${\Star}$) for each $\Star \in \Lambda$, where:
  \begin{enumerate}\itemsep=3pt
  \setlength{\itemindent}{1em}
    \myitem{Sim$_k$} \label{it:sim-prop}
            If $\Simin{w}{v}$ and $w \in P_k$ then $v \in P_k$
    \myitem{Sim${\smile}$} \label{it:sim-smile}
            If $\Simin{w}{v}$ and $w R s$ then there exists a $t \in W$
            such that $v R t$ and $\Simin{t}{s}$
    \myitem{Sim${\frown}$} \label{it:sim-frown}
            If $\Simin{w}{v}$ and $v R t$ then there exists an $s \in W$
            such that $w R s$ and $\Simin{t}{s}$
    \myitem{Sim${\wsmile}$} \label{it:sim-wsmile}
            If $\Simin{w}{v}$ and $v R t$ then \\
            \phantom{\quad}
            \begin{tabular}{|l}
              either $\Simin{v}{t}$,  \\
              or both $\Simin{v}{w}$ and there exists some $s \in W$ such that $w R s$ and $ (s,t)\in\Sim$
              \end{tabular}
    \myitem{Sim${\wfrown}$} \label{it:sim-wfrown}
            If $\Simin{w}{v}$ and $v R t$ then \\
            \phantom{\quad}
            \begin{tabular}{|l}
                either $\Simin{t}{v}$, \\ 
              or
              there exists some $s \in W$ such that $w R s$ and $\Simin{t}{s}$
            \end{tabular}
    \myitem{Sim${\vsmile}$} \label{it:sim-vsmile}
            If $\Simin{w}{v}$ and $w R s$ then \\
            \phantom{\quad}
            \begin{tabular}{|l}
              either $\Simin{w}{s}$, \\
              or
              there exists some $t \in W$ such that $v R t$ and $ \Simin{t}{s}$
            \end{tabular}
    \myitem{Sim${\vfrown}$} \label{it:sim-vfrown}
            If $\Simin{w}{v}$ and $w R s$ then \\
            \phantom{\quad}
            \begin{tabular}{|l}
              either $\Simin{s}{w}$, \\
              or both $\Simin{v}{w}$ and there exists some $t \in W$ such that $v R t$ and $ \Simin{s}{t}$
            \end{tabular}
  \end{enumerate}
  We write $w \simul_{\Lambda} v$ if there exists a $\Lambda$-simulation
  $\Sim$ on $\mo{M}$ such that $\Simin{w}{v}$, and we call $\simul_{\Lambda}$ the relation of \emph{$\Lambda$-similarity}.
\end{definition}

\begin{figure}[h!]
  \centering
    \begin{tikzpicture}[scale=.9, arrows=-latex]
        \node (w) at (0,0) {$w$};
        \node (v) at (2,0) {$v$};
        \node (s) at (0,1.5) {$s$};
        \node (t) at (2,1.5) {$t$};
        \draw[-Circle] (w) to node[below]{\fns{$\Sim$}} (v); 
        \draw (w) to node[left]{\fns{$R$}} (s);
        \draw[densely dotted] (v) to node[right]{\fns{$R$}} (t);
        \draw[densely dotted, -Circle] (t) to node[above]{\fns{$\Sim$}} (s);
        \node at (1,-.85) {\eqref{it:sim-smile}};
        \node (w) at (0,-3.5) {$w$};
        \node (v) at (2,-3.5) {$v$};
        \node (s) at (0,-2) {$s$};
        \node (t) at (2,-2) {$t$};
        \draw[-Circle] (w) to node[below]{\fns{$\Sim$}} (v); 
        \draw[densely dotted] (w) to node[left]{\fns{$R$}} (s);
        \draw (v) to node[right]{\fns{$R$}} (t);
        \draw[densely dotted, -Circle] (t) to node[above]{\fns{$\Sim$}} (s);
        \node at (1,-4.35) {\eqref{it:sim-frown}};
        \node (w) at (4,0) {$w$};
        \node (v) at (6,0) {$v$};
        \node (s) at (4,1.5) {$s$};
        \node (t) at (6,1.5) {$t$};
        \draw[-Circle] (w) to node[below]{\fns{$\Sim$}} (v); 
        \draw (v) to node[right]{\fns{$R$}} (t);
        \draw[dashed, bend left=35,-Circle] (v) to node[left,pos=.7]{\fns{$\Sim$}} (t);
        \draw[densely dotted, bend right=30,-Circle] (v) to node[above,pos=.6]{\fns{$\Sim$}} (w);
        \draw[densely dotted] (w) to node[left]{\fns{$R$}} (s);
        \draw[densely dotted, -Circle] (s) to node[above]{\fns{$\Sim$}} (t);
        \node at (5,-.85) {\eqref{it:sim-wsmile}};
        \node (w) at (4,-3.5) {$w$};
        \node (v) at (6,-3.5) {$v$};
        \node (s) at (4,-2) {$s$};
        \node (t) at (6,-2) {$t$};
        \draw[-Circle] (w) to node[below]{\fns{$S$}} (v); 
        \draw (v) to node[right]{\fns{$R$}} (t);
        \draw[densely dotted] (w) to node[left]{\fns{$R$}} (s);
        \draw[dashed, bend right=35, -Circle] (t) to node[left]{\fns{$\Sim$}} (v);
        \draw[densely dotted, -Circle] (t) to node[above]{\fns{$\Sim$}} (s);
        \node at (5,-4.35) {\eqref{it:sim-wfrown}};
        \node (w) at (8,0) {$w$};
        \node (v) at (10,0) {$v$};
        \node (s) at (8,1.5) {$s$};
        \node (t) at (10,1.5) {$t$};
        \draw[-Circle] (w) to node[below]{\fns{$\Sim$}} (v);
        \draw (w) to node[left]{\fns{$R$}} (s);
        \draw[dashed, bend right=35, -Circle] (w) to node[right]{\fns{$\Sim$}} (s);
        \draw[densely dotted] (v) to node[right]{\fns{$R$}} (t);
        \draw[densely dotted, -Circle] (t) to node[above]{\fns{$\Sim$}} (s);
        \node at (9,-.85) {\eqref{it:sim-vsmile}};
        \node (w) at (8,-3.5) {$w$};
        \node (v) at (10,-3.5) {$v$};
        \node (s) at (8,-2) {$s$};
        \node (t) at (10,-2) {$t$};
        \draw[-Circle] (w) to node[below]{\fns{$S$}} (v);
        \draw (w) to node[left]{\fns{$R$}} (s);
        \draw[dashed, bend left=35, -Circle] (s) to node[right,pos=.3]{\fns{$\Sim$}} (w);
        \draw[densely dotted, bend right=30, -Circle] (v) to node[above,pos=.4]{\fns{$\Sim$}} (w);
        \draw[densely dotted] (v) to node[right]{\fns{$R$}} (t);
        \draw[densely dotted, -Circle] (s) to node[above]{\fns{$\Sim$}} (t);
        \node at (9,-4.35) {\eqref{it:sim-vfrown}};
    \end{tikzpicture}
    \caption{The simulation conditions can be depicted in diagrammatic form. The modal accessibility relation is
    depicted with an arrow, while the simulation relation is drawn as a bullet-headed arrow.
             Solid arrows indicate universal quantification,
             and the dashed and dotted arrows existential ones.
             The dashed and dotted lines should be read as a disjunction:
             either the connections indicated by the dotted lines hold,
             or the connection indicated by the dashed line holds.}
    \label{fig:sim}
  \end{figure}
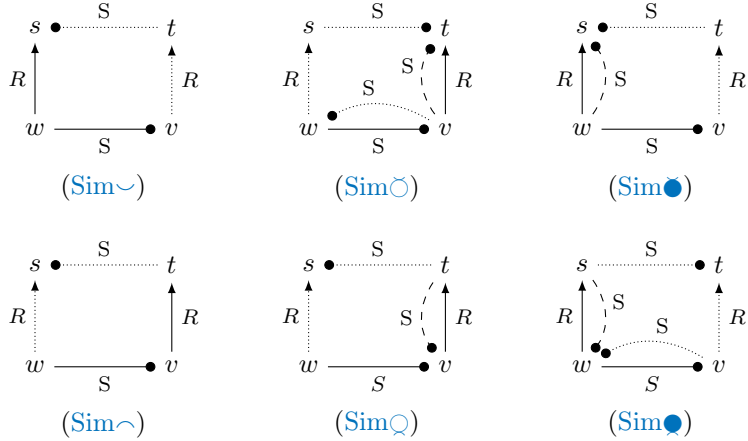

  The simulation conditions are depicted in Figure~\ref{fig:sim}.
  In the diagrams corresponding to~\eqref{it:sim-vsmile} and~\eqref{it:sim-wfrown}, the connections represented by the dotted lines reproduce exactly the simulation conditions for~$\smile$ and~$\frown$, respectively. The restoration conditions are obtained by adding, as an alternative, the $\Sim$-connection represented by the dashed line.
  This reflects the compositional nature of the semantics given to these restoration modalities (recall De\-fi\-ni\-tion~\ref{sat:simil}): $\vsmile\varphi$ combines $\varphi$ and~$\smile\varphi$ conjunctively, whereas $\wfrown\varphi$ combines $\varphi$ and~$\frown\varphi$ disjunctively. A~related compositional pattern occurs for~\eqref{it:sim-wsmile} and~\eqref{it:sim-vfrown}, with respect to the simulation conditions for~$\Box$ and~$\Diamond$ (which can be found in~\cite{KurRij97}).
  We return to this compositional pattern in Section~\ref{sec:combined}.

The conditions for a simulation must find an equilibrium: they need to be weak enough to allow for an intrinsic characterization result, yet strong enough to ensure the preservation of truth between linked worlds.
While the dotted conditions are indeed strong enough to prove that simulations preserve truth of formulas, they are too strong for an intrinsic characterization result.
To achieve the desired intrinsic characterization, we must therefore weaken the simulation conditions, that is, the simulation conditions must be easier to satisfy. 
This is the purpose of adding the dashed arrows. In fact, the dashed arrows in the diagrams for \eqref{it:sim-wsmile}, \eqref{it:sim-vsmile}, \eqref{it:sim-wfrown} and~\eqref{it:sim-vfrown} are required to prove intrinsic characterization results in Section~\ref{sec:hm}. This is showcased in Example~\ref{exm:dashed}, where we give a finite model where the intrinsic characterization result fails if we omit these arrows.
  
  Next, we prove 
  an \emph{invariance} (or \emph{adequacy}) theo\-rem that guarantees that simulation implies subsumption.

\begin{lemma}\label{lem:adeq}
  Let $\Lambda$ be a restorative similarity type, $\mo{M} = \struc{W, R, \{P_{k}\}_{k\in K}}$ be a Kripke model,
  and~$\Sim$ be a $\Lambda$-simulation on~$\mo{M}$.
  Then, for every formula~$\varphi$ in~$\lan_{\Lambda}$ and for all $w,v\in W$ such that $\Simin{w}{v}$, 
  \begin{equation*}
    w\Vdash\varphi\quad\text{implies}\quad
      v\Vdash\varphi.
  \end{equation*}
\end{lemma}
\begin{proof}
We do a structural induction on~$\varphi\in\lan_{\Lambda}$. 
If $\phi$ is either~$\top$ or~$\bot$ then the statement trivially holds.
The case where~$\phi$ is a propositional letter~$p_k$ follows immediately from the fact that $\Sim$ satisfies De\-fi\-ni\-tion~\ref{def:sim}\eqref{it:sim-prop} for every $k \in K$, and the cases where $\phi$ is either of the form $\psi \wedge \chi$ or of the form $\psi \vee \chi$ are routine. So we are left with the cases of the connectives belonging to the restorative similarity type~$\Lambda$.
\medskip

\noindent
  \textit{Case 1: $\phi=\smile\psi$ and $\smile \in \Lambda$.}
  Fix $w,v\in W$ such that $\Simin{w}{v}$ and suppose $w\Vdash\smile\psi$.
  By De\-fi\-ni\-tion~\ref{sat:simil}\eqref{it:sat:smile}, we may now obtain an~$s\in W$ such that $wRs$ and $s\not\Vdash\psi$. 
  By De\-fi\-ni\-tion~\ref{def:sim}\eqref{it:sim-smile}, given that $\Simin{w}{v}$ and $wRs$, we may obtain a~$t\in W$ such that $vRt$ and $\Simin{t}{s}$. From $\Simin{t}{s}$ and $s\not\Vdash\psi$, the induction hypothesis allows us to conclude that $t\not\Vdash\psi$.  Given that $vRt$, using  De\-fi\-ni\-tion~\ref{sat:simil}\eqref{it:sat:smile} again it follows that $v\Vdash\smile\psi$.
\medskip

\noindent
  \textit{Case 2: $\phi=\frown\psi$ and $\frown \in \Lambda$.}
  Fix $w,v\in W$ such that $\Simin{w}{v}$ and suppose by contraposition that $v\not\Vdash\frown\psi$.
  By \eqref{it:sat:frown}, we may obtain a $t\in W$ such that $vRt$ and $t\Vdash\psi$.
  By \eqref{it:sim-frown}, given $\Simin{w}{v}$ and $vRt$, we may obtain an $s\in W$ such that $wRs$ and $\Simin{t}{s}$.
  From $\Simin{t}{s}$ and $t\Vdash\psi$, the induction hypothesis allows us to conclude that $s\Vdash\psi$.  Given that $wRs$, using again \eqref{it:sat:frown} it follows that $w\not\Vdash\frown\psi$.
\medskip

\noindent
  \textit{Case 3: $\phi=\wsmile\psi$ and $\wsmile \in \Lambda$.}
  Fix $w,v\in W$ such that $\Simin{w}{v}$, and suppose by contraposition that $v\not\Vdash\wsmile\psi$.
  By~\eqref{it:sat:wsmile} we know that $v\Vdash\psi$ and we may obtain a $t\in W$ such that $vRt$ and $t\not\Vdash\psi$.
  Using the induction hypothesis we know that $\Simnin{v}{t}$, given that~$v\Vdash\psi$ yet $t\not\Vdash\psi$.
  Then, by~\eqref{it:sim-wsmile}, given that $\Simin{w}{v}$ and $vRt$ and $\Simnin{v}{t}$, we conclude that $\Simin{v}{w}$ and we may obtain an $s\in W$ such that $wRs$ and $\Simin{s}{t}$.
  Also by the induction hypothesis, given $\Simin{v}{w}$ and $v\Vdash\psi$, we have $w\Vdash\psi$.
  Again by the induction hypothesis, given $\Simin{s}{t}$ and $t\not\Vdash\psi$, we have $s\not\Vdash\psi$.
  Taking into account that $wRs$, that $w\Vdash\psi$ and that $s\not\Vdash\psi$, by~\eqref{it:sat:wsmile} we conclude that $w\not\Vdash\wsmile\psi$.
\medskip

\noindent
  \textit{Case 4: $\phi=\wfrown\psi$ and $\wfrown \in \Lambda$.}
  Fix $w,v\in W$ such that $\Simin{w}{v}$, and suppose by contraposition that $v\not\Vdash\wfrown\psi$.
  By~\eqref{it:sat:wfrown} we know that $v\not\Vdash\psi$ and we may obtain a $t\in W$ such that $vRt$ and $t\Vdash\psi$.
  Using the induction hypothesis we know that $\Simnin{t}{v}$, given that~$t\Vdash\psi$ yet $v\not\Vdash\psi$.
  Then, by~\eqref{it:sim-wfrown}, given that $\Simin{w}{v}$ and $vRt$ and $\Simnin{t}{v}$, 
  we may obtain an $s\in W$ such that $wRs$ and $\Simin{t}{s}$.
  Also by the induction hypothesis, given $\Simin{t}{s}$ and $t\Vdash\psi$, we have $s\Vdash\psi$.
  Again by the induction hypothesis, given $\Simin{w}{v}$ and $v\not\Vdash\psi$, we have $w\not\Vdash\psi$.
  Taking into account that $wRs$, that $w\not\Vdash\psi$ and that $s\Vdash\psi$, by~\eqref{it:sat:wfrown} we conclude that $w\not\Vdash\wfrown\psi$.
\medskip

\noindent
  \textit{Case 5: $\phi=\vsmile\psi$ and $\vsmile \in \Lambda$.}
  Fix $w,v\in W$ such that $\Simin{w}{v}$, and suppose $w\Vdash\vsmile\psi$.
  By~\eqref{it:sat:vsmile} we know that $w\Vdash\psi$ and we may obtain an $s\in W$ such that $wRs$ and $s\not\Vdash\psi$.
  Using the induction hypothesis we know that $\Simnin{w}{s}$, given that~$w\Vdash\psi$ yet $s\not\Vdash\psi$.
  Then, by~\eqref{it:sim-vsmile}, given that $\Simin{w}{v}$ and $wRs$ and $\Simnin{w}{s}$, 
  we may obtain a $t\in W$ such that $vRt$ and $\Simin{t}{s}$.
  Also by the induction hypothesis, given $\Simin{t}{s}$ and $s\not\Vdash\psi$, we have $t\not\Vdash\psi$.
  Again by the induction hypothesis, given $\Simin{w}{v}$ and $w\Vdash\psi$, we have $v\Vdash\psi$.
  Taking into account that $vRt$, that $v\Vdash\psi$ and that $t\not\Vdash\psi$, by~\eqref{it:sat:vsmile} we conclude that $v\Vdash\vsmile\psi$.
\medskip

\noindent
  \textit{Case 6: $\phi=\vfrown\psi$ and $\vfrown \in \Lambda$.}
  Fix $w,v\in W$ such that $\Simin{w}{v}$, and suppose $w\Vdash\vfrown\psi$.
  By~\eqref{it:sat:vfrown} we know that $w\not\Vdash\psi$ and we may obtain an $s\in W$ such that $wRs$ and $s\Vdash\psi$.
  Using the induction hypothesis we know that $\Simnin{s}{w}$, given that~$s\Vdash\psi$ yet $w\not\Vdash\psi$.
  Then, by~\eqref{it:sim-vfrown}, given that $\Simin{w}{v}$ and $wRs$ and $\Simnin{s}{w}$, 
  we conclude that $\Simin{v}{w}$ and
  we may obtain a $t\in W$ such that $vRt$ and $\Simin{s}{t}$.
  Also by the induction hypothesis, given $\Simin{v}{w}$ and $w\not\Vdash\psi$, we have $v\not\Vdash\psi$.
  Again by the induction hypothesis, given $\Simin{s}{t}$ and $s\Vdash\psi$, we have $t\Vdash\psi$.
  Taking into account that $vRt$, that $v\not\Vdash\psi$ and that $t\Vdash\psi$, by \eqref{it:sat:vfrown} we conclude that $v\Vdash\vfrown\psi$.
\end{proof}

\begin{theorem}[Adequacy]\label{thm:adeq}
  Let $\Lambda$ be a restorative similarity type.
  Then for all $w, v \in W$ of a Kripke model $\mo{M}$, 
  \begin{equation*}
    w \simul_{\Lambda} v
      \quad\text{implies}\quad
      w \subs_{\Lambda} v.
  \end{equation*}
\end{theorem}
\begin{proof}
Assuming there is some $\Lambda$-simulation~S such that $\Simin{w}{v}$, it immediately follows from the preceding lemma that~$v$ subsumes~$w$.
\end{proof}

  We illustrate how Theo\-rem~\ref{thm:adeq} can provide important information on the expressivity of a given modal language by proving two undefinability results.
  The idea we use for this is that if something were definable, then its truth would be preserved by simulations.

\begin{figure}[h]
  \centering
    \begin{tikzpicture}[scale=.9, arrows=-latex]
        \node (w) at (-1,0)   {$w$};
        \node (s) at (-1,1.5) {$u$};
        \node (v) at ( 1,0)   {$v$};
        \draw[bend left=0] (w) to node[left]{\fns{$R$}} (s);
        \draw[-Circle] (w) to (v);
        \draw[-Circle, bend right=45] (w) to (s);
        \draw[-,fill=blue!10]
              (6,.75) circle(.55);
        \node (w) at (4,.75)   {$w$};
        \node (v) at (6,.75)   {$v$};
        \draw[-Circle] (w) to (v);
    \end{tikzpicture}
  \caption{Depiction of the models and simulations from
           Examples~\ref{exm:smile-vsmile} and~\ref{exm:neg-new}.
           Here, the bullet-headed arrows indicate a connection
           in the relation $\Sim$.
           The highlighted area gives information on the truth sets of the propositional letters: in the left picture these truth sets are all empty; in the right picture we highlight the truth set of the propositional letter $p_\star$, all other truth sets being empty again.}
  \label{fig:exm-adeq}
\end{figure}
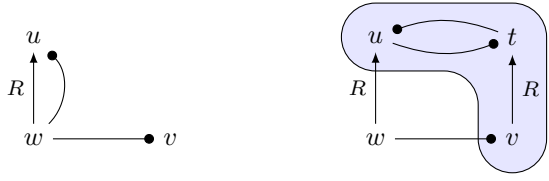

\begin{example}\label{exm:smile-vsmile}
  Consider the Kripke model $\mo{M} = \struc{W, R, \{P_{k}\}_{k\in K}}$ with
  $W = \{ w, v, u \}$, where $(w,u)$ is the only pair of worlds connected by the accessibility relation~$R$, and where $P_k= \emptyset$ for all
  $k \in K$. This is depicted in Figure~\ref{fig:exm-adeq} (left).
  We claim that $\Sim = \{ (w, u), (w, v) \}$ is a $\{ \vsmile \}$-simulation.
  Clearly~\eqref{it:sim-prop} holds for every $k\in K$.
  To see that~\eqref{it:sim-vsmile} holds as well, note that both in the case of $\Simin{w}{u}$ and in the case of $\Simin{w}{v}$ the condition \eqref{it:sim-vsmile} is satisfied by $\Simin{w}{u}$.
  We can use this simulation to prove
  that $\smile$
  is not definable in $\lan_{\vsmile}$.
  Take an arbitrary $k\in K$.
  Observe that we would then have
  $\mo{M}, w \Vdash \smile p_k$ because
  $P_k = \emptyset$. But $\mo{M}, v \not\Vdash \smile p_k$ because $v$
  does not have any successors. Therefore the truth of $\smile p_k$ is not preserved by
  $\{ \vsmile \}$-simulations, and since $\{ \vsmile \}$-simulations preserve the
  truth of all formulas in $\lan_{\vsmile}$,
  this shows that $\smile p_k$ is not equivalent to any formula in $\lan_{\vsmile}$.
\end{example}

\begin{example}\label{exm:neg-new}
  We can use a small example to show that a classical negation is not
  definable in $\lan_{\Lambda}$ for any restorative similarity type
  $\Lambda \subseteq \{ \smile, \frown, \wsmile, \wfrown, \vsmile, \vfrown \}$.
  This extends and simplifies existing results that were recalled in Section~\ref{subsec:negation}.
  Fix $\star\in K$ and
  consider the Kripke model $\mo{M} = \struc{W, R, \{ P_k \}_{k \in K}}$
  where $W = \{ w, v \}$, $R$ is the empty relation, 
  $P_\star = \{ v \}$ and $P_\ell=\varnothing$ for every $\ell\neq\star$.
  Consider the relation $\Sim = \{ (w, v) \}$.
  Then $\Sim$ clearly satisfies \eqref{it:sim-prop},
  and all the remaining conditions from Definition~\ref{def:sim} are
  vacuously satisfied due to the lack of $R$-connections.
  Hence $\Sim$ is a $\Lambda$-simulation for every restorative similarity type~$\Lambda$.

  Note that $\mo{M}, w \not\Vdash p_\star$, while $\mo{M}, v \Vdash p_\star$. 
  Should a classical negation~$\classneg$ be available, we would therefore have that $\mo{M}, w \Vdash \classneg p_\star$ and $\mo{M}, v \not\Vdash \classneg p_\star$.
  Since $(w, v) \in \Sim$, it follows that the truth of $\classneg p_\star$ is not preserved
  by $\Lambda$-simulations. 
  Hence, by Theorem~\ref{thm:adeq},
  $\classneg p_\star$ is not equivalent to any formula in~$\lan_{\Lambda}$.
\end{example}


\section{Simulations between models}\label{sec:sim-between}

\noindent
  So far we have defined simulations ranging over a single model.
  In view of De\-fi\-ni\-tion~\ref{def:disjunion}, a simple way to define simulations \textit{between} models is the following:

\begin{definition}\label{def:sim-between}
  A \emph{$\Lambda$-simulation} between two Kripke models
  $\mo{M}_1$ and $\mo{M}_2$ is a $\Lambda$-simulation on
  $\mo{M}_1 \uplus \mo{M}_2$.
  As usual, we shall identify worlds of $\mo{M}_1$ and $\mo{M}_2$ with their canonical copies in the disjoint union.
\end{definition}

  The use of disjoint unions for defining (bi)simulations between models is not new (see for example \cite[De\-fi\-ni\-tion~3.3]{FanWanDit14}).
  Given models $\struc{W_1, R_1, \{P_{k,1}\}_{k\in K}}$ and $\struc{W_2, R_2, \{P_{k,2}\}_{k\in K}}$, a (bi)simulation from the former to the latter is usually formulated as a relation
  on $W_1 \times W_2$ that satisfies certain conditions.
  When we try to take this path, we run into two problems:
  \begin{enumerate}
    \item We sometimes need preservation of falsehood.
          For example, if a world $w_1 \in W_1$ satisfies $\vsmile\phi$ then it has
          an $R_1$-successor $v_1$ where $\phi$ is false.
          If there is a simulation linking $w_1$ to $w_2 \in W_2$,
          then there needs to be an $R_2$-successor $v_2$ of $w_2$ that is linked
          through the simulation to $v_1$,
          because otherwise preservation of $\vsmile$ may fail.
          This can be achieved by using a second relation that is a simulation
          from $\mo{M}_2$ to $\mo{M}_1$.
          This naturally yields the notion of a \emph{directed simulation},
          given by a pair consisting of
          a \emph{forward relation} $F \subseteq W_1 \times W_2$ and a
          \emph{backward relation} $B \subseteq W_2 \times W_1$, used
          in e.g.~\cite[Section~5]{KurRij97} as well as in
          Definition~\ref{def:dir-sim} below.
    \item The simulation conditions for the restoration modalities require
          $\Sim$-connections between worlds of the same model (Example~\ref{exm:dashed} illustrates their necessity).
          This means that we do not just require relations \emph{between}
          $W_1$ and $W_2$, but we also need to express
          simulation relations \emph{on} $W_1$ and \emph{on} $W_2$.
          Therefore, on top of our forward and backward relations,
          we also need relations on each of the individual models to be part of our notion of
          simulation.
  \end{enumerate}
  Summarizing,
  if we want to describe a simulation between two Kripke models
  without viewing it as a relation on their disjoint union,
  then we need to use four different relations, namely
  a relation $\Sim_{ij} \subseteq W_i \times W_j$ for each $i, j \in \{ 1, 2 \}$.
  In order for such a quadruple $(\Sim_{11}, \Sim_{12}, \Sim_{21}, \Sim_{22})$ to be a
  $\Lambda$-simulation, each of the $\Sim_{ij}$ should satisfy analogues
  of the simulation conditions (Sim${\Star}$), where $\Star \in \Lambda$.
  These analogues are obtained by replacing the $\Simin{w}{v}$ in the premise of
  the condition by $(w,v) \in \Sim_{ij}$, and then replacing each occurrence of
  $\Sim$ with $\Sim_{ij}$ and each occurrence of $R$ with $R_1$ or $R_2$
  in such a way that the expression is well typed.
  
  Consider for example (\Sim$_{12}\wsmile$).
  The premise becomes \guillemotleft If $\Simin[12]{w}{v}$ and $v R_2 t$\guillemotright,
  because $\Simin[12]{w}{v}$ implies that $v \in W_2$. So, for $v R t$ to make
  sense, $R$ should be the relation from $\mo{M}_2$.
  This is then followed by \guillemotleft$\Simin[22]{v}{t}$\guillemotright\ because both $v$ and $t$ are in $W_2$.
  Continuing the process, this yields
  \begin{enumerate}
    \setlength{\itemindent}{1em}
    \myitem{\Sim$_{12}\wsmile$}
            If $\Simin[12]{w}{v}$ and $v R_2 t$ then either \\
            \phantom{\quad}
            \begin{tabular}{|l}
              $\Simin[22]{v}{t}$, \\
              or $\Simin[21]{v}{w}$ and there exists some $s \in W_1$ such that $w R_1 s$ and $\Simin[12]{s}{t}$
            \end{tabular}
  \end{enumerate}
  Accordingly, we introduce the following notion of a concrete $\Lambda$-simulation
  between Kripke models, which avoids explicit use of the disjoint union.
  
\begin{definition}
  A \emph{concrete $\Lambda$-simulation} between Kripke models
  $\mo{M}_1 = \struc{W_1, R_1, \{P_{k,1}\}_{k\in K}}$ and $\mo{M}_2 = \struc{W_2, R_2, \{P_{k,2}\}_{k\in K}}$
  consists of a quadruple $(\Sim_{11}, \Sim_{12}, \Sim_{21}, \Sim_{22})$ of
  relations such that $\Sim_{ij} \subseteq W_i \times W_j$ for each $i, j \in \{ 1, 2 \}$,
  which satisfy (\Sim$_{ij}\Star$) for each $i, j \in \{ 1, 2 \}$ and each
  $\Star \in \Lambda$.
\end{definition}

  Such a notion of simulation is no more than a reformulation of a $\Lambda$-simulation
  on the disjoint union of models:

\begin{proposition}
  Two worlds in two models are related by a concrete $\Lambda$-simulation
  if and only if they are related by a $\Lambda$-simulation.
\end{proposition}
\begin{proof}
  If $(\Sim_{11}, \Sim_{12}, \Sim_{21}, \Sim_{22})$ is a concrete $\Lambda$-simulation
  between $\mo{M}_1$ and $\mo{M}_2$, then $\Sim_{11} \cup \Sim_{12} \cup \Sim_{21} \cup \Sim_{22}$ is a $\Lambda$-simulation
  on $\mo{M}_1 \uplus \mo{M}_2$. Conversely, any $\Lambda$-simulation
  $\Sim$ on $\mo{M}_1 \uplus \mo{M}_2$ gives rise to a concrete $\Lambda$-simulation
  where $\Sim_{ij} := \Sim \cap (W_i \times W_j)$ for each $i, j \in \{ 1, 2 \}$.
\end{proof}

  The duplication of the simulation conditions not only makes the de\-fi\-ni\-tion
  of a concrete $\Lambda$-simulation much longer, but it also fails to reward
  with extra insight. 
  Therefore, we prefer in general to work with $\Lambda$-simulations
  and to define simulations between models as in Definition~\ref{def:sim-between}.
  
  However, in some cases we can cut down the amount of relations needed to explicitly express
  similarity between Kripke models.
  This happens when $\Lambda \subseteq \{ \smile, \frown \}$, because
  the simulation conditions for these modalities do not have a dashed arrow.
  The next de\-fi\-ni\-tion brings to mind work done on simulations for the positive modalities
  $\Box$ and $\Diamond$ in~\cite{KurRij97}. However, since we are working
  with negative modalities our simulation conditions require both
  preservation of truth and preservation of falsehood (i.e.\ reflection of truth).
  Therefore, when defining a notion of simulation between 
  Kripke models $\mo{M}_1 = \struc{W_1, R_1, \{P_{k,1}\}_{k\in K}}$ and $\mo{M}_2 = \struc{W_2, R_2, \{P_{k,2}\}_{k\in K}}$,
  we are forced to work with pairs of relations consisting of
  a \emph{forward relation} $F \subseteq W_1 \times W_2$ and
  a \emph{backward relation} $B \subseteq W_2 \times W_1$.
  This indeed resembles the directed intuitionistic bisimulations
  used in~\cite[Section~5]{KurRij97}.

\begin{definition}\label{def:dir-sim}
  Let $\Lambda \subseteq \{ \smile, \frown \}$.
  A \emph{directed $\Lambda$-simulation} between Kripke models
  $\mo{M}_1 = \struc{W_1, R_1, \{P_{k,1}\}_{k\in K}}$ and $\mo{M}_2 = \struc{W_2, R_2, \{P_{k,2}\}_{k\in K}}$ consists of a pair
  $(F, B)$ of relations $F \subseteq W_1 \times W_2$ and $B \subseteq W_2 \times W_1$
  that satisfies both~\eqref{it:forw-prop} and~\eqref{it:back-prop}, for every $k \in K$,
  and also both (F${\Star}$) and (B${\Star}$), for each $\Star \in \Lambda$, where:
  \begin{enumerate}
  \setlength{\itemindent}{1em}
    \myitem{F$_k$} \label{it:forw-prop}
            If $(w_1,w_2)\in F$ and $w_1 \in P_{k,1}$, then $w_2 \in P_{k,2}$ for all $k\in K$
    \myitem{B$_k$} \label{it:back-prop}
            If $(w_2,w_1)\in B$ and $w_2 \in P_{k,2}$, then $w_1 \in P_{k,1}$ for all $k\in K$
    \myitem{F${\smile}$} \label{it:forw-smile}
            If $(w_1,w_2)\in F$ and $w_1 R_1 v_1$, then there exists a $t_2 \in W_2$
            such that $w_2 R_2 t_2$ and $(t_2,v_1)\in B$
    \myitem{B${\smile}$} \label{it:back-smile}
            If $(w_2,w_1)\in B$ and $w_2 R_2 v_2$, then there exists a $t_1 \in W_1$
            such that $w_1 R_1 t_1$ and $(t_1,v_2)\in F$
    \myitem{F${\frown}$} \label{it:forw-frown}
            If $(w_1,w_2)\in F$ and $w_2 R_2 v_2$, then there exists a $t_1 \in W_1$
            such that $w_1 R_1 t_1$ and $(v_2,t_1)\in B$
    \myitem{B${\frown}$} \label{it:back-frown}
            If $(w_2,w_1)\in B$ and $w_1 R_1 v_1$, then there exists a $t_2 \in W_2$
            such that $w_2 R_2 t_2$ and $(v_1,t_2)\in F$
  \end{enumerate}
  We write $w \psimul_{\Lambda} v$ to indicate that there exists a directed $\Lambda$-simulation
  $(F, B)$ such that $(w, v) \in F$, and we call $\psimul_{\Lambda}$ the relation of \emph{directed $\Lambda$-similarity}.
\end{definition}

  We note that $(w, v) \in F$ does not necessarily entail $(v, w) \in B$,
  and similarly $(v, w) \in B$ does not entail $(w, v) \in F$.
  Furthermore, observe that the directed simulation conditions are essentially the same as
  the two left-hand conditions in Figure~\ref{fig:sim},
  except that we are now using two relations: one linking
  worlds in $\mo{M}_1$ to worlds in $\mo{M}_2$ and one linking worlds in
  the opposite direction.
  The next proposition proves that for appropriate $\Lambda$,
  the notions of $\Lambda$-similarity and directed $\Lambda$-similarity coincide.

\begin{proposition}\label{prop:dsim}
  Let $\mo{M}_1 = \struc{W_1, R_1, \{P_{k,1}\}_{k\in K}}$ and $\mo{M}_2 = \struc{W_2, R_2, \{P_{k,2}\}_{k\in K}}$ be
  two Kripke models and let $w_1 \in W_1$ and $w_2 \in W_2$.
  Suppose $\Lambda \subseteq \{ \smile, \frown \}$.
  Then $w_1 \simul_{\Lambda} w_2$ if and only if $w_1 \psimul_{\Lambda} w_2$.
\end{proposition}
\begin{proof}
  Suppose $w_1 \simul_{\Lambda} w_2$. It follows that there exists a $\Lambda$-simulation $\Sim$ on $\mo{M}_1 \uplus \mo{M}_2$ such that
  $\Simin{w_1}{w_2}$.
  Now take $\Forw := \Sim \cap (W_1 \times W_2)$ and
  $\Back := \Sim \cap (W_2 \times W_1)$. A~simple verification then shows that $(\Forw, \Back)$ is a directed $\Lambda$-simulation.
  Since $w_1 \in W_1$ and $w_2 \in W_2$, we have $w_1 \Forw w_2$,
  hence $w_1 \psimul_{\Lambda} w_2$.
  The converse direction follows from the fact that for any
  directed simulation $(\Forw, \Back)$ between $\mo{M}_1$ and
  $\mo{M}_2$, the union $\Sim := \Forw \cup \Back$ is a $\Lambda$-simulation
  on $\mo{M}_1\uplus \mo{M}_2$, where $\Back^{-1} := \{ (x, y) \mid (y, x) \in \Back \}$.
\end{proof}

  Combining Theo\-rem~\ref{thm:adeq} and Proposition~\ref{prop:dsim} yields
  invariance for directed $\Lambda$-simulations,
  where $\Lambda \subseteq \{ \smile, \frown \}$.
  In other words, $w_1 \psimul_{\Lambda} w_2$ implies $w_1 \subs_{\Lambda} w_2$.

\section{Intrinsic characterization results}\label{sec:hm}

\noindent
  We have seen that the existence of a $\Lambda$-simulation implies subsumption.
  A straightforward adaptation of~\cite[Example~2.23]{BlaRijVen01} shows, however, that the
  converse does not hold in general.
  As is customary, we can show the
  converse holds for the class of image-finite Kripke models, and the result further extends to the class of modally saturated Kripke models.

  For the cases of $\smile$ and $\frown$ the proof is analogous to the classical
  or the positive case --- see e.g.~\cite[Theo\-rem~2.24]{BlaRijVen01}
  and~\cite[Proposition~3.4]{KurRij97}.
  However, the simulation conditions for the restoration modalities contain
  additional $\Sim$-connections (the dashed arrows in Figure~\ref{fig:sim}),
  which reflect the fact that these modalities are, in a certain sense, made up of $\smile\phi$
  or $\frown\phi$ together with truth or falsehood of $\phi$.
  The additional $\Sim$-connections provide an alternative way of satisfying a simulation condition. They are essential to the proof of the intrinsic characterization results in this section because the absence of such connections allows one to assume the
  existence of a formula $\chi$ that is true in one world and false in another.
  In Example~\ref{exm:dashed} we illustrate the failure of the intrinsic characterization result when we omit these dashed arrows from the simulation conditions. 

  We now introduce the notions of satisfiability and refutability. The latter is going to be particularly useful for dealing with negative modalities.
  \begin{definition}\label{def:sat-ref}
  Let $\mo{M} = \struc{W, R, \{P_{k}\}_{k\in K}}$ be a Kripke model and $U \subseteq W$.
  A set $\Delta$ of formulas is said to be
  \emph{satisfiable in $U$}
  if there exists a world $u \in U$ such
  that $\mo{M}, u \Vdash \phi$ for all $\phi \in \Delta$. 
  Analogously, $\Delta$ is said to be \emph{refutable in $U$}
  if there exists a world $u \in U$ such
  that $\mo{M}, u \not\Vdash \phi$ for all $\phi \in \Delta$.
  The set $\Delta$ is called \emph{finitely satisfiable in $U$} 
  if
  every finite subset of $\Delta$ is satisfiable in $U$, and
  \emph{finite refutability} is defined analogously.
  \end{definition}

   We also recall the de\-fi\-ni\-tions of image-finiteness and of modal sa\-tu\-ra\-tion~\cite{BlaRijVen01}, a compactness property that will be of interest to us:%
  \begin{definition}\label{def:mod-sat}
    Let $\mo{M} = \struc{W, R, \{ P_k \}_{k \in K} }$ be a Kripke model.
    We call $\mo{M}$ \emph{image-finite} if $R[w]$ is finite for every $w \in W$. 
    We call $\mo{M}$ \emph{modally saturated}
  if for every set $\Delta \subseteq \lan_{\classneg,\Box,\Diamond}$ of formulas and
  all $w \in W$,
  we have that $\Delta$ is satisfiable in $R[w]$ 
  whenever it is finitely satisfiable in $R[w]$.
\end{definition}

In what follows, where $\Delta$ is a finite set of formulas, we will use $\bigwedge\Delta$ and $\bigvee\Delta$ to denote the conjunction and disjunction of all the formulas in~$\Delta$, respectively. As usual, if $\Delta$ is a singleton $\{\delta\}$, then $\bigwedge\Delta = \bigvee\Delta = \delta$, and if $\Delta$ is empty then $\bigwedge\Delta$ is defined as~$\top$ and $\bigvee\Delta$ is defined as~$\bot$.
In the proofs of the following lemmas, we shall freely make use of the satisfaction conditions of the connectives (Definition \ref{sat:simil}), without explicitly referring to them.

\begin{lemma}\label{lem:hm-finite}
  Let $\Lambda$ be a restorative similarity type
  and let $\mo{M} = \struc{W, R, \{P_{k}\}_{k\in K}}$ be an image-finite Kripke model.
  Define $\Sim$ as the subsumption relation,
  that is, set $(w, v) \in \Sim$ iff $w \subs_{\Lambda} v$.
  Then $\Sim$ satisfies \textup{(Sim${\Star}$)} for each $\Star \in \Lambda$.
\end{lemma}
\begin{proof}
  \textit{Case 1: $\Star = \smile$ and $\smile \in \Lambda$.}
    Assume (i) $\Simin{w}{v}$ and (ii) $w R s$, and suppose towards a contradiction that (iii) there exists no $t \in W$ such that $v R t$ and $\Simin{t}{s}$.  
    Thus, from (iii) and the definition of $\Sim$, it follows that (iv) given any $u \in R[v]$ we may obtain a formula~$\varphi_u$ such that $u \Vdash \varphi_u$ while $s \not\Vdash \varphi_u$.  Let~$\varphi$ be $\bigvee \{ \varphi_u \mid u \in R[v] \}$.  Image-finiteness guarantees that this is well defined.  
    Now, fix an arbitrary $t\in R[v]$. Then, from (iv) we have in particular that $t \Vdash \varphi_t$, so 
    we know that (v) $t \Vdash \varphi$.  
    Moreover, from (iv) we also have that (vi) $s \not\Vdash \varphi_t$.
    Given that $t$ is an arbitrary $R$-successor of~$v$,
    from (v) and (vi), respectively, we conclude that (vii) $v \not\Vdash \smile\varphi$ and (viii) $s\not\Vdash\varphi$.  From (viii), in view of assumption (ii),  
    it follows that (ix) $w \Vdash \smile\varphi$.  Given the definition of~$\Sim$, we note that (ix) and (vii) contradict assumption~(i).

  \medskip\noindent
  \textit{Case 2: $\Star = \frown$ and $\frown \in \Lambda$.}
  Assume (i) $\Simin{w}{v}$ and (ii) $v R t$, and suppose towards a contradiction that (iii) there exists no world $s$ such that $w R s$ and $\Simin{t}{s}$.  Thus, from (iii) and the definition of~$\Sim$ it follows that (iv) given any $u\in R[w]$ we may obtain a formula $\varphi_u$ such that $t\Vdash\varphi_u$ while $u\not\Vdash\varphi_u$.  Let $\varphi$ be $\bigwedge \{ \phi_u \mid u \in R[w] \}$. Image-finiteness guarantees that this is well defined. Now, fix an arbitrary $s\in R[w]$.  Then from (iv) we have in particular that $s\not\Vdash\varphi_s$, so we know that (v) $s\not\Vdash\varphi$.  Moreover, from (iv) we also have that (vi) $t\Vdash\varphi_s$.  Given that~$s$ is an arbitrary $R$-successor of~$w$, from (v) and (vi), respectively, we conclude that (vii) $w\Vdash\frown\varphi$ and (viii) $t\Vdash\varphi$.  From (viii), in view of assumption (ii), it follows that (ix) $v\not\Vdash\frown\varphi$.  Given the definition of~$\Sim$, we note that (vii) and (ix) contradict assumption (i).
 
  \medskip\noindent
  \textit{Case 3: $\Star = \wsmile$ and $\wsmile \in \Lambda$.}
    Suppose $\Simin{w}{v}$ and $v R t$.
    If $\Simin{v}{t}$ then we are done, so suppose that this is not the case.
    Then by de\-fi\-ni\-tion of $\Sim$ there exists some formula $\chi$
    such that $v \Vdash \chi$ and $t \not\Vdash \chi$.
    Now we prove two things:
    \begin{enumerate}
      \item[(a)] $\Simin{v}{w}$;
      \item[(b)] there exists a world $s \in W$ such that $w R s$ and $\Simin{s}{t}$.
    \end{enumerate}
    We prove them one by one:
    \begin{enumerate}
    \item[(a)] Suppose towards a contradiction that $(v, w) \notin \Sim$.
    Then by de\-fi\-ni\-tion of $\Sim$ there exists a formula $\xi$
    such that $v \Vdash \xi$ while $w \not\Vdash \xi$.
    This implies that $w \not\Vdash \chi \wedge \xi$ and
    $v \Vdash \chi \wedge \xi$ and $t \not\Vdash \chi \wedge \xi$.
    The former entails $w \Vdash \wsmile(\chi \wedge \xi)$ while the
    latter two give $v \not\Vdash \wsmile(\chi \wedge \xi)$.
    This contradicts the assumption that $\Simin{w}{v}$,
    i.e.~that $w$ is subsumed by $v$.

    \item[(b)] Suppose towards a contradiction that there exists no world
    $s \in R[w]$ such that $\Simin{s}{t}$.
    Then for each such $s$ we can find a formula $\psi_s$ such that
    $s \Vdash \psi_s$ while $t \not\Vdash \psi_s$.
    By assumption $R[w]$ is finite, so we can define
    $\psi := \bigvee \{ \psi_s \mid s \in R[w] \}$.
    Then we have $s \Vdash \psi$ for all $s \in R[w]$, hence clearly
    also $s \Vdash \psi \vee \chi$ for all $s \in R[w]$.
    Therefore $w \Vdash \wsmile(\psi \vee \chi)$.
    On the other hand, we find $t \not\Vdash \psi$ and hence
    $t \not\Vdash \psi \vee \chi$. Moreover, the fact that $v \Vdash \chi$
    implies $v \Vdash \psi \vee \chi$. Combining this gives
    $v \not\Vdash \wsmile(\psi \vee \chi)$.
    Again, we find a contradiction with the fact that $w$ is subsumed
    by $v$. So our assumption must be false, and we can find some
    world $s$ such that $wRs$ and $\Simin{s}{t}$, as desired.
    \end{enumerate}
    
  \medskip\noindent
  \textit{Case 4: $\Star = \wfrown$ and $\wfrown \in \Lambda$.}
    Suppose $\Simin{w}{v}$ and $v R t$.
    If $\Simin{t}{v}$ then we are done, so suppose that this is not the case.
    Then by de\-fi\-ni\-tion of $\Sim$ there exists some formula $\chi$
    such that $t \Vdash \chi$ and $v \not\Vdash \chi$.

    Now suppose towards a contradiction that there exists no $s \in R[w]$
    such that $\Simin{t}{s}$. Then for each such~$s$ we can find a formula
    $\psi_s$ such that $t \Vdash \psi_s$ and $s \not\Vdash \psi_s$.
    Let $\psi = \bigwedge \{ \psi_s \mid s \in R[w] \}$. Then we have $t \Vdash \psi \wedge \chi$
    and $v \not\Vdash \psi \wedge \chi$, so $v \not\Vdash \wfrown(\psi \wedge \chi)$. But also no $s \in R[w]$ satisfies $\psi \wedge \chi$,
    wherefore $w \Vdash \wfrown(\psi \wedge \chi)$.
    So, $\Simin{w}{v}$ implies $v \Vdash \wfrown(\psi \wedge \chi)$,
    a contradiction.

  \medskip\noindent
  \textit{Case 5: $\Star = \vsmile$ and $\vsmile \in \Lambda$.}
    Suppose $\Simin{w}{v}$ and $w R s$.
    If $\Simin{w}{s}$ then we are done, so suppose that this is not the case.
    Then by de\-fi\-ni\-tion of $\Sim$ there exists some formula $\chi$
    such that $w \Vdash \chi$ and $s \not\Vdash \chi$.
    Now suppose towards a contradiction that there exists no $t \in R[v]$
    such that $\Simin{t}{s}$. Then for each such $t$ we can find a formula
    $\psi_t$ such that $t \Vdash \psi_t$ and $s \not\Vdash \psi_t$.
    Let $\psi = \bigvee \{ \psi_t \mid t \in R[v] \}$.
    Then $t \Vdash \psi \vee \chi$ for all $t \in R[v]$.
    Furthermore, we have $s \not\Vdash \psi\lor\chi$ and $w \Vdash \psi\lor\chi$,
    and hence $w\Vdash \vsmile (\psi\lor\chi)$.
    By de\-fi\-ni\-tion of $\Sim$ we find $v\Vdash \vsmile(\psi\lor\chi)$,
    so there must be a $t \in R[v]$ such that $t \not \Vdash \psi \vee \chi$,
    contradicting the fact that all $t \in R[v]$ satisfy $\psi \vee \chi$.

  \medskip\noindent
  \textit{Case 6: $\Star = \vfrown$ and $\vfrown \in \Lambda$.}
    Suppose $\Simin{w}{v}$ and $w R s$.
    If $\Simin{s}{w}$ then we are done, so suppose that this is not the case.
    Then by de\-fi\-ni\-tion of $\Sim$ there exists some formula $\chi$
    such that $s \Vdash \chi$ and $w \not\Vdash \chi$.
    Now we prove two things:
    \begin{enumerate}
      \item[(a)] $\Simin{v}{w}$;
      \item[(b)] there exists a world $t \in W$ such that $v R t$ and $\Simin{s}{t}$.
    \end{enumerate}
    We prove them one by one:
    \begin{enumerate}
    \item[(a)] Suppose towards a contradiction that $(v, w) \notin \Sim$.
    Then by de\-fi\-ni\-tion of $\Sim$ there exists a formula $\xi$
    such that $v \Vdash \xi$ while $w \not\Vdash \xi$.
    This implies that $v \Vdash \chi \vee \xi$ and
    $w \not\Vdash \chi \vee \xi$ and $s \Vdash \chi \vee \xi$.
    The former entails $v \not\Vdash \vfrown(\chi \vee \xi)$ while the
    latter two give $w \Vdash \vfrown(\chi \vee \xi)$.
    This contradicts the assumption that $\Simin{w}{v}$,
    i.e.~that $w$ is subsumed by $v$.
    \item[(b)] Suppose towards a contradiction that there exists no world
    $t \in R[v]$ such that $\Simin{s}{t}$.
    Then for each such $t$ we can find a formula $\psi_t$ such that
    $s \Vdash \psi_t$ while $t \not\Vdash \psi_t$.
    By assumption $R[v]$ is finite, so we can define
    $\psi := \bigwedge \{ \psi_t \mid t \in R[v] \}$.
    Then we have $t \not\Vdash \psi$, hence $t \not\Vdash \psi \wedge \chi$, 
    for all $t \in R[v]$.
    Therefore $v \not\Vdash \vfrown(\psi \wedge \chi)$.
    On the other hand, we find $s \Vdash \psi$ and hence
    $s \Vdash \psi \wedge \chi$. Moreover, the fact that $w \not\Vdash \chi$
    implies $w \not\Vdash \psi \wedge \chi$. Combining this gives
    $w \Vdash \vfrown(\psi \wedge \chi)$.
    Again, we find a contradiction with the fact that $w$ is subsumed
    by $v$. So our assumption must be false, and we can find some
    world $t$ such that $vRt$ and $\Simin{s}{t}$, as desired.
  \end{enumerate}
  This completes the proof of the lemma.
\end{proof}

\begin{theorem}\label{thm:hm-finite}
  Let $\Lambda$ be a restorative similarity type and
  $\mo{M} = \struc{W, R, \{P_{k}\}_{k\in K}}$ be an image-finite Kripke model. Then
  for all $w, v \in W$ we have
  \begin{equation*}
    w \subs_{\Lambda} v \iff w \simul_{\Lambda} v.
  \end{equation*}
\end{theorem}
\begin{proof}
  The direction from right to left follows from Theorem~\ref{thm:adeq}.
  The direction from left to right follows from the fact that,
  according to Lemma~\ref{lem:hm-finite},
  $\subs_{\Lambda}$ is a $\Lambda$-simulation.
\end{proof}

  The following example demonstrates the need for the $\Sim$-connections
  running parallel to an $R$-connection in the simulation conditions for
  the restoration modalities, i.e.\ the need for the dashed arrows from
  Figure~\ref{fig:sim}. Specifically, it demonstrates that if we
  omit such connections from~\eqref{it:sim-wsmile},
  then $w \subs_{\wsmile} v$ no longer guarantees $w \simul_{\wsmile} v$.

\begin{example}\label{exm:dashed}
  Let us work in the language $\lan_{\wsmile}$, and
  suppose the dashed arrow is omitted from the simulation condition
  for $\wsmile$, i.e.~suppose that condition is simplified to: 
\begin{enumerate}\itemsep=3pt
  \setlength{\itemindent}{1em}
    \myitem{Sim${\wsmile}$}\label{it:sim-wsmile-prime}
            If $\Simin{w}{v}$ and $v R t$ then \\
            \phantom{\quad}
            \begin{tabular}{|l}
              $\Simin{v}{w}$ and there exists some $s \in W$ such that $w R s$ and $ (s,t)\in\Sim$
              \end{tabular}

  \end{enumerate}
  Consider the Kripke model $\mo{M} = \struc{W, R, \{P_{k}\}_{k\in K}}$ with
  $W = \{ w, v, t \}$, $R = \{ (v, t) \}$ and $P_k = \{ t \}$ for every $k \in K$,
  depicted in Figure~\ref{fig:exm-hm}.
  Then, according to \eqref{it:sim-wsmile-prime}, any simulation relation $\Sim$ such that $(w, v) \in \Sim$ requires
  the existence of some world $u$ such that $wRu$. Since there is no such
  $u$ in the model, $w$ and $v$ cannot be related by a simulation.
  
  However, we claim that $v$ does subsume $w$. 
  To see this, first observe that $t$ satisfies every formula in $\lan_{\wsmile}$
  except those equivalent to $\bot$.
  So, for each formula $\phi \in \lan_{\wsmile}$, either $v \not\Vdash \phi$ (if $\phi$ is equivalent to $\bot$)
  or $t \Vdash \phi$. Hence, $v \Vdash \wsmile\phi$.
  A routine induction on the structure of $\phi$ now proves that $w \Vdash \phi$ implies $v \Vdash \phi$, for all $\phi \in \lan_{\wsmile}$,
  so that indeed $w \subs_{\wsmile} v$.
  Thus, the simplified simulation condition~\eqref{it:sim-wsmile-prime}
  results in a failure of the intrinsic characterization result for
  $\lan_{\wsmile}$. Similar examples can be found for each of the
  other restoration modalities.
\end{example}

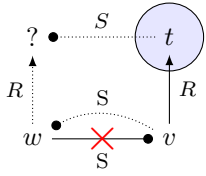
\begin{figure}
  \centering
    \begin{tikzpicture}[scale=.9, arrows=-latex]
        \draw[-,fill=blue!10]
              (6,1.5) circle(.5);
        \node (w) at (4,0)   {$w$};
        \node (v) at (6,0)   {$v$};
        \node (t) at (6,1.5) {$t$};
        \node (u) at (4,1.5) {?};
        \draw (v) to node[right]{\fns{$R$}} (t);
        \draw[-Circle] (w) to node[below,yshift=-2pt]{\fns{$\Sim$}}
                              node{\LARGE\textcolor{red}{$\times$}} (v);
        \draw[densely dotted,-Circle, bend right=30] (v) to node[above]{\fns{$\Sim$}} (w);
        \draw[densely dotted] (w) to node[left]{\fns{$R$}} (u);
        \draw[densely dotted,-Circle] (t) to node[above]{\fns{$S$}} (u);
    \end{tikzpicture}
  \caption{Depiction of the model from Example~\ref{exm:dashed}.
           The highlighted area shows the truth sets of all propositional letters.
           The bullet-headed arrow indicates our desired simulation connection,
           and the dotted arrows indicate the required connections (that we do not have).}
  \label{fig:exm-hm}
\end{figure}

  Our next goal is to generalize the previous theo\-rem to a larger class
  of Kripke models, namely for the modally saturated Kripke models.
  Since modal saturation is defined for the language $\lan_{\classneg,\Box,\Diamond}$, and all modalities
  we consider may be viewed as abbreviations in this language,
  it immediately follows that in any modally saturated Kripke model $\mo{M} = \struc{W, R, \{P_{k}\}_{k\in K}}$
  and for any world $w \in W$, and any restorative similarity type $\Lambda$,
  finite satisfiability (resp.\ refutability) of a set $\Delta$ of $\lan_{\Lambda}$-formulas
  in $R[w]$ implies satisfiability (resp.\ refutability) in $R[w]$:

\begin{lemma}
  Let $\Lambda$ be a restorative similarity type and
  $\mo{M} = \struc{W, R, \{P_{k}\}_{k\in K}}$ a modally saturated Kripke model.
  Then for all $w \in W$ and all $\Delta \subseteq \lan_{\Lambda}$,
  if $\Delta$ is finitely refutable in $R[w]$, then it is refutable in $R[w]$.
\end{lemma}

  We can use (finite) satisfiability and refutability to extend
  Lemma~\ref{lem:hm-finite} to the modally saturated setting.

\begin{lemma}\label{{lem:hm-msat}}
  Let $\Lambda$ be a restorative similarity type and
  $\mo{M} = \struc{W, R, \{P_{k}\}_{k\in K}}$ be a modally saturated Kripke model.
  Define $\Sim$ as the subsumption relation.
  Then $\Sim$ satisfies \textup{(Sim${\Star}$)} for each $\Star \in \Lambda$.
\end{lemma}
\begin{proof}
  In contrast to the proof of Lemma~\ref{lem:hm-finite}, which relies on image-finiteness to obtain finite conjunctions and disjunctions, here we use modal saturation for the same effect. %
  We showcase three of the six cases, the others being similar.
  
  \medskip\noindent
  \textit{Case 1: $\Star = \smile$ and $\smile \in \Lambda$.}
    Let $w, v, s \in W$ be such that $\Simin{w}{v}$ and $w R s$.
    Let $\Delta := \{ \phi \in \lan_{\Lambda} \mid s \not\Vdash \phi \}$.
    Then $s \not\Vdash \bigvee\Delta_0$,
    hence $w \Vdash \smile(\bigvee\Delta_0)$, for each finite $\Delta_0 \subseteq \Delta$.
    Since $\Simin{w}{v}$, this implies that $v \Vdash \smile(\bigvee\Delta_0)$ for every finite
    $\Delta_0 \subseteq \Delta$, and therefore every finite subset of $\Delta$
    is refutable in $R[v]$.
    Since $\mo{M}$ is modally saturated, $\Delta$ must be refutable in $R[v]$.
    This implies that there exists some $t \in R[v]$ that refutes every formula
    that is refuted by $s$. This implies $t \subs_{\Lambda} s$, hence $\Simin{t}{s}$,
    as required.
 
  \medskip\noindent
  \textit{Case 3: $\Star = \wsmile$ and $\wsmile \in \Lambda$.}
    Let $w, v, t \in W$ be such that $\Simin{w}{v}$ and $v R t$.
    If $\Simin{v}{t}$ then we are done, so suppose that $(v,t)\not\in \Sim$. 
    Then by the definition of~$\Sim$ we may obtain a formula $\chi \in \lan_\Lambda$ that is true at $v$ but not at $t$.
    Now we prove two things:
    \begin{enumerate}
      \item[(a)] $\Simin{v}{w}$;
      \item[(b)] there exists a world $s \in W$ such that $w R s$ and $\Simin{s}{t}$.
    \end{enumerate}
    Item (a) proceeds as in Case 3 of Lemma~\ref{lem:hm-finite}, so we prove (b):
    \begin{enumerate}
      \item[(b)]
        Let $\Delta := \{ \phi \in \lan_{\Lambda} \mid t \not\Vdash \phi \}$.
        Then $\chi \in \Delta$, so $t \not\Vdash \chi \vee \bigvee \Delta_0$
        for any finite $\Delta_0 \subseteq \Delta$.
        Since $v \Vdash \chi$, we have $v \Vdash \chi \vee \bigvee\Delta_0$,
        hence $v \not\Vdash \wsmile(\chi \vee \bigvee\Delta_0)$ for any
        finite $\Delta_0 \subseteq \Delta$.
        Since $w \subs_{\Lambda} v$ by assumption, we have
        $w \not\Vdash \wsmile(\chi \vee \bigvee\Delta_0)$ for any finite
        $\Delta_0 \subseteq \Delta$, so that in particular $\Delta_0$ is
        refutable at some successor of $w$.
        This implies that $\Delta$ is finitely refutable in $R[w]$,
        hence by modal saturation there exists some $s \in R[w]$
        that does not satisfy any formula in $\Delta$.
        It follows that $s \subs_{\Lambda} t$, hence $\Simin{s}{t}$, as desired.
    \end{enumerate}
    
  \medskip\noindent
  \textit{Case 4: $\Star = \wfrown$ and $\wfrown \in \Lambda$.}
    Let $w, v, t \in W$ be such that $\Simin{w}{v}$ and $v R t$.
    If $\Simin{t}{v}$ then we are done, so suppose that this is not the case.
    Then by the de\-fi\-ni\-tion of $\Sim$ we may obtain a formula $\chi$ that
    is true at $t$ but not at $v$.
    
    Now let $\Delta := \{ \phi \in \lan_{\Lambda} \mid t \Vdash \phi \}$.
    Then $t \Vdash \chi \wedge \bigwedge\Delta_0$ for any finite $\Delta_0 \subseteq \Delta$.
    But by assumption $v \not\Vdash \chi$,
    so $v \not\Vdash \chi \wedge \bigwedge\Delta_0$,
    and hence $v \not\Vdash \wfrown(\chi \wedge \bigwedge\Delta_0)$.
    Since $w$ is subsumed by $v$ this implies
    $w \not\Vdash \wfrown(\chi \wedge \bigwedge\Delta_0)$,
    so there must be a successor of $w$ that satisfies $\bigwedge\Delta_0$.
    This implies that $\Delta$ is finitely satisfiable in $R[w]$,
    hence satisfiable because $\mo{M}$ is modally saturated.
    So there must be some $s \in R[w]$ that satisfies all formulas in $\Delta$.
    Therefore $t \subs_{\Lambda} s$, so $\Simin{t}{s}$.
\end{proof}

  Analogously to Theo\-rem~\ref{thm:hm-finite}, we now obtain:

\begin{theorem}[Intrinsic characterization]\label{thm:hm-msat}
  Let $\Lambda$ be a restorative similarity type and
  let $\mo{M} = \struc{W, R, \{P_{k}\}_{k\in K}}$ be a modally saturated Kripke model.
  Then for all $w, v \in W$,
  \begin{equation*}
    w \subs_{\Lambda} v \iff w \simul_{\Lambda} v.
  \end{equation*}
\end{theorem}

  In the special case where the restorative similarity type $\Lambda$ contains no restoration modalities, we obtain the following result concerning the directed $\Lambda$-simulations introduced in Definition \ref{def:dir-sim}:
  
\begin{corollary}\label{cor:collapse:sim-dirsim}
  Let $\Lambda \subseteq \{ \smile, \frown \}$ and let
  $\mo{M} = \struc{W, R, \{P_{k}\}_{k\in K}}$ be a modally saturated Kripke model.
  Then for all $w, v \in W$, we have
  $w \subs_{\Lambda} v$ iff $w \psimul_{\Lambda} v$.
\end{corollary}

\section{Relative characterization results}\label{sec:vb}

\noindent
  In this section we prove a relative characterization result for our logics,
  parametric in the restorative similarity type $\Lambda$. Informally, this result states that
  the corresponding restorative modal logic can be viewed as 
  the fragment of first-order logic with one free variable that is preserved by $\Lambda$-simulations.
  Our proof of it runs along the lines of
  \cite[Theo\-rem~3.5]{KurRij97} and~\cite[Theo\-rem~2.68]{BlaRijVen01}.
  Recall from Definition~\ref{def:languages} that $\{ p_k \mid k \in K \}$ is an arbitrary but fixed set of propositional variables.
  We start by defining a suitable first-order language:

\begin{definition}\label{fol:sign}
  Consider a first-order signature~$\Sigma$ containing a binary predicate symbol $\mathbf{R}$ and a unary predicate symbol $\mathbf{P}_k$ for each $k \in K$.
  Let $\Var$ be a set of variables. 
  We shall use $\FOL$ to refer to first-order classical logic with equality, defined over $\Sigma$ and an infinite set of variables $\Var$.
\end{definition}

  Note that the first-order structures for $\FOL$ correspond precisely to
  Kripke models: the interpretation of the binary predicate symbol is given by
  the accessibility relation~$R$, while the interpretations of the unary predicate symbols represent the truth sets of propositional letters.
  Accordingly, we henceforth identify Kripke models and $\FOL$-structures. 
  There should be, however, no risk of confusion: 
  in writing $\mo{M},w\Vdash\varphi$, we will be assuming that $\varphi$ is a modal formula and referring to satisfaction at the world~$w$ of the Kripke model~$\mo{M}$. In writing $\mo{M}\Vdash\alpha[g]$, we will be assuming that $\alpha$ is a first-order formula and referring to the usual notion of first-order satisfaction in the structure~$\mo{M}$ under the assignment~$g$. When $\alpha(x)$ has at most the free variable~$x$, we abbreviate by $\mo{M}\Vdash\alpha(x)[w]$ the case in which $g(x)=w$.

\begin{definition}\label{def:st}
  For each $\Lambda \subseteq \{ \smile, \frown, \wsmile, \wfrown, \vsmile, \vfrown \}$,
  the standard translation $\st : \Var\times\lan_{\Lambda} \to \FOL$ is defined recursively by setting 
  \begin{align*}
    \st(x, p_k) &:= \mathbf{P}_kx \\
    \st(x,\top)
      &:= (x = x)
      &\st(x, \phi \wedge \psi)
      &:= \st(x, \phi) \wedge \st(x, \psi) \\
    \st(x, \bot)
      &:= \classneg(x = x)
    &\st(x, \phi \vee \psi)
      &:= \st(x, \phi) \vee \st(x, \psi) \\[0mm]
    \st(x, \smile\phi)
      &:= (\exists y)(x \mathbf{R} y \wedge \classneg \st(y, \phi))
      &\st(x, \frown\phi)
      &:= (\forall y)(x \mathbf{R} y \to \classneg\st(y,\phi)) \\
    \st(x, \wsmile\phi)
      &:= \classneg\st(x, \phi) \vee (\forall y)(x \mathbf{R} y \to \st(y, \phi))
      &\st(x, \wfrown\phi)
      &:= \st(x, \phi) \vee (\forall y)(x \mathbf{R} y \to \classneg \st(y, \phi)) \\
    \st(x, \vsmile\phi)
      &:= \st(x, \phi) \wedge (\exists y)(x \mathbf{R} y \wedge \classneg\st(y, \phi))
      &\st(x, \vfrown\phi)
      &:= \classneg\st(x, \phi) \wedge (\exists y)(x \mathbf{R} y \wedge \st(y, \phi))
  \end{align*}
where, in each modal clause, $y$ is chosen fresh.
\end{definition}

  A routine induction on the structure of $\phi$ then allows us to verify that the standard translation indeed \guillemotleft captures\guillemotright\ the satisfaction clauses stated in Definition~\ref{sat:simil}, in the following precise sense:

\begin{lemma}\label{lem:st}
  For every Kripke model $\mo{M} = \struc{W, R, \{P_{k}\}_{k\in K}}$, every $w \in W$ and every
  $\phi \in \lan_\Lambda$, we have
  \begin{equation*}
    \mo{M}, w \Vdash \phi \iff \mo{M} \Vdash \st(x,\phi)[w].
  \end{equation*}
\end{lemma}

Recall from Definition~\ref{def:sim-between} that a $\Lambda$-simulation
between $\mo{M}$ and $\mo{M}'$ is a $\Lambda$-simulation on
$\mo{M}\uplus\mo{M}'$, identifying worlds with their canonical copies in
the disjoint union.
  Preservation of $\FOL$-formulas 
  under $\Lambda$-simulations is defined as expected:

\begin{definition}
  We say that the $\FOL$-formula $\alpha(x)$ with one free variable $x$ is
  \emph{preserved by $\Lambda$-simulations} if for every $\Lambda$-simulation
  $\Sim$ between Kripke models $\mo{M} = \struc{W, R, \{P_{k}\}_{k\in K}}$
  and $\mo{M}' = \struc{W', R', \{ P_k' \}_{k \in K}}$ such that $(w, w') \in \Sim$,
  \begin{equation*}
    \mo{M} \Vdash \alpha(x)[w] \quad\text{implies}\quad \mo{M}'\Vdash \alpha(x)[w'].
  \end{equation*}
\end{definition}

  The above definition is designed to supply, within the \guillemotleft modal fragment of first-order logic\guillemotright, a notion serving as the analogue of homomorphisms between first-order structures.
  Now we can use the intrinsic characterization results for
  $\Lambda$-simulations to obtain results about the expressive strength of restorative modal logics:

\begin{theorem}[Relative characterization]\label{thm:vb}
  Let $\Lambda$ be a restorative similarity type.
  Let $\alpha(x)$ be a $\FOL$-formula with one free variable $x$.
  Then $\alpha(x)$ is equivalent to the standard translation of
  an $\lan_{\Lambda}$-formula if and only if it is preserved by
  $\Lambda$-simulations.
\end{theorem}
\begin{proof}
  If $\alpha(x)$ is equivalent to the standard translation of some
  $\lan_{\Lambda}$-formula, then it follows from Theo\-rem~\ref{thm:adeq} and
  Lemma~\ref{lem:st} that it is preserved by $\Lambda$-simulations.
  So let us assume that $\alpha(x)$ is preserved by $\Lambda$-simulations.
We use~$\models$ for the standard notion of entailment in first-order logic.
  Consider the set 
  \begin{equation*}
    \mathcal{K}(x) := \{ \st(x,\psi) \mid \psi \in \lan_{\Lambda} \text{ and } \alpha(x) \models \st(x,\psi)\}
  \end{equation*}
  It suffices to show that $\mathcal{K}(x) \models \alpha(x)$,
  because then compactness for first-order logic entails the existence of
  $\st(x, \psi_1), \ldots, \st(x, \psi_n) \in \mathcal{K}(x)$ such that
  $\st(x, \psi_1) \wedge \cdots \wedge \st(x, \psi_n) \models \alpha(x)$,
  so that $\st(x, \psi_1 \wedge \cdots \wedge \psi_n) \models \alpha(x)$.
  Since $\alpha(x) \models \st(x, \psi_i)$ for each $i \in \{ 1, \ldots, n \}$,
  it then follows that $\alpha(x) \models \st(x, \psi_1 \wedge \cdots \wedge \psi_n)$,
  making $\alpha(x)$ equivalent to the standard translation
  of $\psi_1 \wedge \cdots \wedge \psi_n$.

  So, let $\mo{M} = \struc{W, R, \{P_{k}\}_{k\in K}}$ be a Kripke model, let $w \in W$ and assume
  $\mo{M} \Vdash \mathcal{K}(x)[w]$. We aim to show that $\mo{M} \Vdash \alpha(x)[w]$. 
  Let
  \begin{equation*}
    \mathcal{T}(x) = \{ \classneg\st(x, \psi) \mid \psi \in \lan_{\Lambda} \text{ and } \mo{M} \not\Vdash \st(x,\psi)[w]\}
  \end{equation*}
  We claim that $\mathcal{T}(x) \cup \{ \alpha(x) \}$ is consistent.
  Suppose not, then by compactness of first-order logic there exist
  $\classneg\st(x, \psi_1), \ldots, \classneg\st(x, \psi_n) \in \mathcal{T}(x)$ such that
  $\alpha(x) \models \classneg(\classneg\st(x, \psi_1) \wedge \cdots \wedge \classneg\st(x, \psi_n))$,
  hence
  \begin{equation*}
    \alpha(x) \models \st(x, \psi_1) \vee \cdots \vee \st(x, \psi_n).
  \end{equation*}
  This implies that
  $\st(x, \psi_1) \vee \cdots \vee \st(x, \psi_n) = \st(x, \psi_1 \vee \cdots \vee \psi_n) \in \mathcal{K}(x)$,
  so $\mo{M} \Vdash \st(x, \psi_1)[w] \vee \cdots \vee \st(x, \psi_n)[w]$.
  But then $\mo{M} \Vdash \st(x, \psi_i)[w]$ for some $i \in \{ 1, \ldots, n \}$,
  contradicting the assumption that $\classneg\st(x, \psi_i) \in \mathcal{T}(x)$.
  Therefore $\mathcal{T}(x) \cup \{ \alpha(x) \}$ must be consistent.

  Since $\mathcal{T}(x) \cup \{ \alpha(x) \}$ is consistent, there must exist a Kripke model $\mo{N}$ and a world $v$ such that
  $\mo{N} \Vdash \classneg\st(x, \psi)[v]$ for every $\classneg\st(x, \psi) \in \mathcal{T}(x)$,
  and $\mo{N} \Vdash \alpha(x)[v]$.
  Therefore, by construction of $\mathcal{T}(x)$, if $\mo{N} \Vdash \st(x, \phi)[v]$ then $\mo{M} \Vdash \st(x, \phi)[w]$,
  for all $\phi \in \lan_{\Lambda}$, hence, by Lemma~\ref{lem:st}, $w$ subsumes $v$.
  
  Let $\mo{N}^*$ and $\mo{M}^*$ be $\omega$-saturated elementary extensions
  of $\mo{N}$ and $\mo{M}$, respectively, and let $v^*$ and $w^*$ be the
  images of $v$ and $w$~\cite[Theorems~4.1.9, 6.1.4 and 6.1.8]{ChaKei73}.
  Then $\mo{N}^*, v^* \subs_{\Lambda} \mo{M}^*, w^*$.
  Moreover, since every $\omega$-saturated model is modally saturated
  (see e.g.~\cite[Theorem~2.65]{BlaRijVen01}),
  Theorem~\ref{thm:hm-msat} yields $\mo{N}^*, v^* \simul_{\Lambda} \mo{M}^*, w^*$.
  Finally, since $\mo{N} \Vdash \alpha(x)[v]$ we have
  $\mo{N}^* \Vdash \alpha(x)[v^*]$ and because $\alpha(x)$ is preserved by
  $\Lambda$-simulations we get $\mo{M}^* \Vdash \alpha(x)[w^*]$.
  Invariance of truth of classical first-order formulas under elementary embeddings then gives
  $\mo{M} \Vdash \alpha(x)[w]$, as desired.
\end{proof}

  Since $\Lambda$-similarity coincides with directed $\Lambda$-similarity
  when $\Lambda \subseteq \{ \smile, \frown \}$ (recall Proposition~\ref{prop:dsim}), we also get:

\begin{corollary}
  Let $\Lambda \subseteq \{ \smile, \frown \}$ and
  let $\alpha(x)$ be a $\FOL$-formula with one free variable $x$.
  Then $\alpha(x)$ is equivalent to the standard translation of
  a $\lan_{\Lambda}$-formula if and only if it is preserved by
  directed $\Lambda$-simulations.
\end{corollary}
\begin{proof}
  Combine Proposition~\ref{prop:dsim} with Theorem~\ref{thm:vb}.
\end{proof}

\section{Symmetrical simulations for restorative modal logics with a classical negation}\label{sec:symm}

\noindent
    The main logics considered so far in this work had their languages induced by restorative similarity types.
    By adding a classical negation to some of those logics they become exactly as expressive as \guillemotleft standard\guillemotright\ modal logics (the logics having $\{\classneg,\Box,\Diamond\}$ as their similarity type).
    For example, if one simply adds a classical negation to a restorative modal logic with a diamond-minus connective, an appropriate notion of bisimulation is already known~\cite[Section~2.2]{BlaRijVen01}, because $\smile$ and $\Box$ are interdefinable in the presence of classical negation.
    However, this is not the case for every restorative similarity type
    (see Example~\ref{ex:undef:new}),
    so the study of bisimulations does not always boil down to the bisimulations for the language $\lan_{\classneg,\Box,\Diamond}$, which we will refer to as \emph{Kripke bisimulations}~\cite[Section~2.2]{BlaRijVen01}.

    In this section, we extend the results presented for logics with language $\lan_\Lambda$ with $\Lambda\subseteq \{ \smile, \frown, \wsmile, \wfrown, \vsmile, \vfrown \}$ to logics with language $\lan_{\{ \classneg \} \cup \Lambda}$,
    which we denote by $\lan_{\classneg,\Lambda}$.
    In this way, we obtain some results already studied in the literature, as in~\cite{Fan15}, as well as new ones.
    We do so by restricting our focus to \emph{symmetrical}
  $\Lambda$-simulations. 

\begin{definition}
  Let $\Lambda$ be a restorative similarity type.
  A \emph{symmetrical $\Lambda$-simulation} on a Kripke model
  $\mo{M} = \struc{W, R, \{ P_k \}_{k \in K}}$ is a $\Lambda$-simulation
  $\Sim$ on $\mo{M}$ such that
  $\Simin{w}{v}$ implies $\Simin{v}{w}$, for all $w, v \in W$.
  We write $w \symsim_{\Lambda} v$ if there exists a symmetrical $\Lambda$-simulation
  linking $w$ and $v$.
\end{definition}

\begin{remark}
    We extend the definition of \emph{subsumption} for languages $\lan_\Lambda$ from Definition \ref{sat:simil} to languages $\lan_{\classneg,\Lambda}$ as expected:
    \[
        \mo{M}_1, w_1 \subs_{\classneg,\Lambda} \mo{M}_2, w_2
  \iff
  \mo{M}_1, w_1 \Vdash \phi \text{ implies } \mo{M}_2, w_2 \Vdash \phi, \text{ for every } \phi \in \lan_{\classneg,\Lambda}.
    \]
\end{remark}

\begin{example}\label{ex:sym-normal}
  If $\Lambda$ contains $\smile$ or $\frown$,
  then any symmetrical $\Lambda$-simulation immediately satisfies 
  the usual Kripke bisimulation conditions~\cite[Definition~2.16]{BlaRijVen01}.
  Conversely, every Kripke bisimulation $\mathrm{B}$ such that $\mathrm{B}$ is
  symmetric clearly satisfies all simulation conditions from
  Definition~\ref{def:sim}, so it is a symmetrical $\Lambda$-simulation for any restorative similarity type $\Lambda$.
  Since every Kripke bisimulation can be extended to a symmetrical one, it follows that whenever either $\smile \in \Lambda$ or $\frown \in \Lambda$,
  symmetrical $\Lambda$-bisimilarity and Kripke bisimilarity
  coincide.
  So in these cases the study of symmetrical $\Lambda$-simulations boils down to Kripke bisimulations.
\end{example}

\begin{example}\label{ex:Fan}
  The restorative modal language with $\wsmile$ and a classical negation was introduced in~\cite{Mar04:accident}, where it was shown to be strictly less expressive than $\lan_{\sim,\Box,\Diamond}$.
  Subsequently, in~\cite{Fan15}, a notion of $\circ$-bisimulation for the corresponding logic, as well as
  intrinsic and relative characterization results,
  were provided.
  The notion of a $\circ$-bisimulation is similar to our symmetrical $\{ \wsmile \}$-simulation,
  but they are not exactly the same:
  every symmetrical $\{ \wsmile \}$-simulation is a $\circ$-bisimulation,
  but an arbitrary $\circ$-bisimulation $\mathrm{B}$ is a symmetrical $\{ \wsmile \}$-simulation
  if and only if $\mathrm{B}$ is a symmetric relation.
  Since every $\circ$-bisimulation $\mathrm{B}$ can be extended to a symmetrical $\circ$-bisimulation, its induced notion of bisimilarity coincides with
  symmetrical $\{ \wsmile \}$-similarity.
  In other words, two worlds can be linked by a $\circ$-bisimulation if and only if they
  can be linked by a symmetrical $\{ \wsmile \}$-simulation.
\end{example}

  If $w_1$ and $w_2$ are worlds in two Kripke models $\mo{M}_1$ and $\mo{M}_2$,
  then we write $w_1 \logeq_{\Lambda} w_2$ when they satisfy precisely the
  same $\lan_{\classneg,\Lambda}$-formulas. We note that this coincides with
  logical subsumption in $\lan_{\classneg,\Lambda}$,
  because preservation of the truth of the classical negation of a formula implies reflection of the truth of that formula.
  
  The symmetrical aspect of symmetrical $\Lambda$-simulations ensures that it at once
  preserves and reflects the truth of formulas. That is, worlds linked by a
  symmetrical $\Lambda$-simulation satisfy precisely the same formulas:

\begin{theorem}\label{thm:sym-adeq}
  Let $\Lambda$ be a restorative 
  similarity type and $\mo{M} = \struc{W, R, \{ P_k \}_{k \in K}}$
  a Kripke model. Then for all worlds $w, v \in W$,
  \begin{equation*}
    w \symsim_{\Lambda} v
    \quad\text{implies}\quad
    w \logeq_{\Lambda} v.
  \end{equation*}
\end{theorem}
\begin{proof}
  We use induction on the structure of $\phi$ to prove that for any
  symmetrical $\Lambda$-simulation $\Sim$ on $\mo{M}$ and all worlds $w, v \in W$,
  $w \Vdash \phi$ implies $v \Vdash \phi$.
  This suffices because, in the presence of a classical negation,
  subsumption (i.e.~$w \subs_{\classneg,\Lambda} v$) coincides with logical equivalence ($w \logeq_{\Lambda} v$).
  All cases except for negation proceed as in the proof of
  Lem\-ma~\ref{lem:adeq}. The induction case for negation follows
  immediately from the symmetry of $\Sim$.
\end{proof}

  The intrinsic characterization result for modally saturated Kripke models
  from Theorem~\ref{thm:hm-msat} also readily carries over.
  Using the same proofs as in Section~\ref{sec:hm} one may check that
  $\subs_{\classneg,\Lambda}$ is a $\Lambda$-simulation. But since
  this coincides with $\logeq_{\Lambda}$, it follows that the resulting
  $\Lambda$-simulation is also symmetric. Accordingly:

\begin{theorem}[Symmetrical intrinsic characterization]\label{thm:shm}
  Let $\Lambda$ be a restorative similarity type and
  let $\mo{M} = \struc{W, R, \{P_{k}\}_{k\in K}}$ be a modally saturated Kripke model.
  Then, for all $w, v \in W$,
  \begin{equation*}
    w \logeq_{\Lambda} v \iff w \symsim_{\Lambda} v.
  \end{equation*}
\end{theorem}
  This allows us to derive a relative characterization result via a minor adaptation.
  First, we extend the standard translation from De\-fi\-ni\-tion~\ref{def:st} to
  a translation $\st : \Var \times \lan_{\classneg,\Lambda} \to \FOL$ by adding
  the recursive clause
  \begin{equation*}
    \st(x,\classneg\phi) =\classneg\st(x, \phi).
  \end{equation*}
  Second, we say that a first-order formula $\alpha(x)$ with one free
  variable $x$ is \emph{invariant under symmetrical $\Lambda$-simulations}
  if for every Kripke model $\mo{M} = \struc{W, R, \{ P_k \}_{k \in K}}$,
  all worlds $w, v \in W$, and every symmetrical $\Lambda$-simulation $\Sim$
  such that $\Simin{w}{v}$, we have
  \begin{equation*}
    \mo{M} \Vdash \alpha(x)[w] \iff \mo{M} \Vdash \alpha(x)[v].
  \end{equation*}
  Finally, we generalize Theo\-rem~\ref{thm:vb} to:

\begin{theorem}
  Let $\Lambda$ be a restorative similarity type, and $\alpha(x)$ a $\FOL$-formula
  with one free variable $x$.
  Then $\alpha(x)$ is equivalent to the standard translation of a
  $\lan_{\classneg,\Lambda}$-formula if and only if it is invariant under
  symmetrical $\Lambda$-simulations.
\end{theorem}
\begin{proof}
  The proof is exactly the same as that of Theorem~\ref{thm:vb},
  except that we appeal to the symmetrical intrinsic characterization result (Theorem~\ref{thm:shm})
  instead of its counterpart from Theorem~\ref{thm:hm-msat}.
\end{proof}

  This theorem allows us to recover (a slight variation of) the characterization
  result from \cite[Theorem~29]{Fan15}. Moreover, it yields new
  characterization results for other modal logics with a classical negation whose
  restorative similarity types are subsets of
  $\{ \smile, \frown, \wsmile, \wfrown, \vsmile, \vfrown \}$.

  Example \ref{ex:sym-normal} considered languages sharing the same expressive power as  $\lan_{\classneg,\Box,\Diamond}$ and showed that, for those languages, the concept of a symmetrical simulation coincides with Kripke bisimulation. 
  But some restorative similarity types give rise to languages
  that are strictly less expressive than
  $\lan_{\classneg,\Box,\Diamond}$. For example, it was shown in \cite{Mar04:accident} that $\Box$ is not definable in the language $\lan_{\wsmile,\classneg}$.
  Using the tools developed in this section, we establish that the languages presented in Example~\ref{ex:sym-normal},
  i.e.~languages of the form $\lan_{\sim,\Lambda}$ where
  $\Lambda$ contains $\smile$ or $\frown$, are the only ones that are equally expressive as $\lan_{\classneg,\Box,\Diamond}$. 
  Consequently, apart from the language $\lan_{\wsmile,\classneg}$ where the concept of bisimulation was introduced in~\cite{Fan15}, the characterization results proved above are all new.

\begin{example}\label{ex:undef:new}
  Consider the Kripke model $\mo{M} = \struc{W, R, \{ P_k \}_{k \in K}}$ where $W = \{ w, v \}$, $R = \{ (w, w) \}$ and $P_k = \emptyset$ for all $k \in K$ (see Figure \ref{fig:count-box-def}).
  Then a routine verification shows that $\Sim = \{ (w, v), (v, w), (w,w) \}$ is a symmetrical $\Lambda$-simulation for any
  $\Lambda \subseteq \{ \wsmile, \vsmile, \wfrown, \vfrown \}$.
  Therefore Theorem \ref{thm:sym-adeq} guarantees that $w$ and $v$ satisfy precisely the same $\lan_{\sim,\Lambda}$-formulas.
  
  Now observe that $w\not\Vdash \Box \bot$ and $v \Vdash \Box \bot$. This implies that $\Box\bot$ is not definable in $\lan_{\sim,\Lambda}$.
  Since $\Box$ is definable whenever we have $\smile$ or $\frown$, it follows that none of $\Box, \Diamond, \smile$ or $\frown$ is definable in $\lan_{\sim,\Lambda}$.
  Furthermore, this demonstrates that the study of simulations for languages of the form $\lan_{\sim,\Lambda}$ with 
  $\Lambda \subseteq \{ \wsmile, \vsmile, \wfrown, \vfrown \}$
  is different from the study of Kripke bisimulations.
    \begin{figure}[h!]
        \centering
        \begin{tikzpicture}[scale=.9, arrows=-latex]
        \node (w) at (0,0) {$w$};
        \node (v) at (2,0) {$v$};
        \draw[-Circle] (w) to node[below]{\fns{$\Sim$}} (v); 
        \draw[-Circle] (v) to node[below]{\fns{$\Sim$}} (w); 
        \draw[-Circle, loop, in=55, out=125, distance=10mm] (w) to node[above]{\fns{$\Sim$}} (w);
        \draw[loop, in=-55, out=-125, distance=10mm] (w) to node[below]{\fns{$R$}} (w);
        \end{tikzpicture}
        \caption{Kripke model $\struc{W, R, \{P_k\}_{k\in K}}$ such that $w,v\in W$, $P_k=\varnothing$ for every $k \in K$, and $R=\{(w,w)\}$, together with a simulation $S=\{(w,v),(v,w),(w,w)\}$}.
        \label{fig:count-box-def}
    \end{figure}
\end{example}

\section{Outlook: combined modalities}\label{sec:combined}

\noindent
  Each of the restoration modalities $\wsmile, \vsmile, \wfrown$ and $\vfrown$
  is equivalent to an affirmation or a classical negation combined with one of
  $\smile, \frown, \Diamond$ and $\Box$.
  It is not fortuitous, indeed, that this gets reflected in the clauses defining the standard translation, in De\-fi\-ni\-tion~\ref{def:st}.
  Furthermore, the simulation conditions of the restoration modalities are
  closely related to the simulation conditions of their \guillemotleft ingredients\guillemotright.
  We briefly sketch a line of research concerning this pattern, starting with
  the notion of a \emph{combined modality}.

\begin{definition}
  Let $\aff$ be a classical-like modal \emph{affirmation} operator, that is, assume that $\aff\phi$ and $\phi$ are true at precisely the same worlds.
  Let $\Omega := \{ \aff, \classneg, \Diamond, \Box, \smile, \frown \}$.
  For each $\Star_1, \Star_2 \in \Omega$, we specify the \emph{combined modalities}
  $\combor{\Star_1}{\Star_2}$ and $\comband{\Star_1}{\Star_2}$ semantically by
  \begin{align*}
    \mo{M}, w \Vdash \comband{\Star_1}{\Star_2}\phi
      &\iff \mo{M}, w \Vdash \Star_1 \phi \text{ and } \mo{M}, w \Vdash \Star_2 \phi \\
    \mo{M}, w \Vdash \combor{\Star_1}{\Star_2}\phi
      &\iff \mo{M},w\Vdash \Star_1 \phi \text{ or } \mo{M}, w \Vdash \Star_2 \phi
  \end{align*}
  We will refer to $\Star_1$ and $\Star_2$ as the \emph{components} of the
  combined modality.
\end{definition}

\begin{example}\label{ex:undef}
  Each of the restoration modalities $\wsmile, \vsmile, \wfrown$ and $\vfrown$
  may be viewed as a combined modality. For example
  the modality $\wsmile$ can be written as $\combor{\classneg}{\Box}$ and $\vsmile$ is the same as $\comband{{\aff}}{\smile}$.
  Additionally, combined modalities encompass the \guillemotleft contingency\guillemotright\
  modality $\triangledown = \comband{\Diamond}{\smile}$
  and the \guillemotleft non-contingency\guillemotright\ modality ${\vartriangle} = \combor{\Box}{\frown}$~\cite{MonRou66,Cre88},
  as well as the \guillemotleft wrong belief\guillemotright\ modality $W = \comband{\classneg}{\Box}$~\cite{Ste11}
  and the \guillemotleft strong accidentality\guillemotright\ modality $\odot = \comband{\aff}{\frown}$~\cite{PanYan17}.
  We note that $\Box$ is not in general definable from $\triangledown$ or $\vartriangle$~\cite{kuhn:1995:non-contingency,humberstone:1995:non-contingency}, nor from $\wsmile$ and $\vsmile$ (the latter having been originally introduced, in \cite{Mar04:accident}, as modalities for \guillemotleft essentiality\guillemotright\ and for \guillemotleft accidentality\guillemotright), even in the presence of a classical-like negation.
\end{example}

\begin{remark}
This perspective is closely related to the recent general theory of \textit{propositional modal bundled modalities} developed in \cite{DingYang2026} (see also \cite{Zeng2026} for a broader systematic study of bundled fragments), where complex modal constructions are packaged as primitive operators. In particular, each of the combined modalities considered in this section may, in isolation, be viewed as a bundled modality in this sense. 
There are, however, important differences in scope. The general framework of~\cite{DingYang2026} assumes a Boolean base language with classical negation and, for simplicity, considers a single primitive bundled modality at a time, whose interpretation is parametrized by a bundle term. Our languages, by contrast, need not contain a classical negation and may contain several negative and restoration modalities simultaneously.  
\end{remark}

  Each of the operators in $\Omega$ comes with a familiar simulation
  condition, depicted in the left column of Figure~\ref{fig:sim} and
  in Figure~\ref{fig:sim-C} below.

\begin{figure}[h!]
  \centering
    \begin{tikzpicture}[scale=.9, arrows=-latex]
        \node (w) at (0,0) {$w$};
        \node (v) at (2,0) {$v$};
        \draw (w) to node[below]{\fns{$\Sim$}} (v);
        \draw[densely dotted, bend left=40] (w) to node[above]{\fns{$\Sim$}} (v);
        \node at (1,-.85) {\Sim${\aff}$};
        \node (w) at (4,0) {$w$};
        \node (v) at (6,0) {$v$};
        \draw (w) to node[below]{\fns{$\Sim$}} (v);
        \draw[densely dotted, bend right=40] (v) to node[above]{\fns{$\Sim$}} (w);
        \node at (5,-.85) {\Sim${\classneg}$};
        \node (w) at (8,0) {$w$};
        \node (v) at (10,0) {$v$};
        \node (s) at (8,1.5) {$s$};
        \node (t) at (10,1.5) {$t$};
        \draw (w) to node[below]{\fns{$\Sim$}} (v); 
        \draw[densely dotted] (v) to node[right]{\fns{$R$}} (t);
        \draw (w) to node[left]{\fns{$R$}} (s);
        \draw[densely dotted] (s) to node[above]{\fns{$\Sim$}} (t);
        \node at (9,-.85) {\Sim${\Diamond}$};
        \node (w) at (12,0) {$w$};
        \node (v) at (14,0) {$v$};
        \node (s) at (12,1.5) {$s$};
        \node (t) at (14,1.5) {$t$};
        \draw (w) to node[below]{\fns{$\Sim$}} (v); 
        \draw (v) to node[right]{\fns{$R$}} (t);
        \draw[densely dotted] (w) to node[left]{\fns{$R$}} (s);
        \draw[densely dotted] (s) to node[above]{\fns{$\Sim$}} (t);
        \node at (13,-.85) {\Sim${\Box}$};
    \end{tikzpicture}
    \caption{The simulation conditions for $\aff, \classneg, \Box$ and $\Diamond$.}
    \label{fig:sim-C}
\end{figure}

  We can now observe a pattern in the de\-fi\-ni\-tion of the simulation conditions
  for the restoration modalities from Definition~\ref{def:sim}: 
  each arises from taking the simulation conditions for the components by adding a \emph{gluing condition}. For example, in the case of $\wsmile = \combor{\classneg}{\Box}$ we have:
  
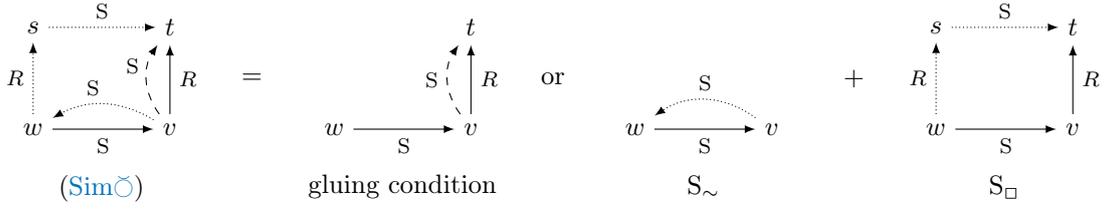
\begin{center}
    \begin{tikzpicture}[scale=.9, arrows=-latex]
        \node (w) at (0,0) {$w$};
        \node (v) at (2,0) {$v$};
        \node (s) at (0,1.5) {$s$};
        \node (t) at (2,1.5) {$t$};
        \draw (w) to node[below]{\fns{$\Sim$}} (v); 
        \draw (v) to node[right]{\fns{$R$}} (t);
        \draw[dashed, bend left=35] (v) to node[left,pos=.7]{\fns{$\Sim$}} (t);
        \draw[densely dotted, bend right=30] (v) to node[above,pos=.6]{\fns{$\Sim$}} (w);
        \draw[densely dotted] (w) to node[left]{\fns{$R$}} (s);
        \draw[densely dotted] (s) to node[above]{\fns{$\Sim$}} (t);
        \node at (1,-.85) {\eqref{it:sim-wsmile}};
        \node at (3.2,.75) {$=$};
        \node (w) at (4.4,0) {$w$};
        \node (v) at (6.4,0) {$v$};
        \node (t) at (6.4,1.5) {$t$};
        \draw (w) to node[below]{\fns{$\Sim$}} (v); 
        \draw (v) to node[right]{\fns{$R$}} (t);
        \draw[dashed, bend left=35] (v) to node[left]{\fns{$\Sim$}} (t);
        \node at (5.4,-.85) {gluing condition};
        \node at (7.6,.75) {or};
        \node (w) at (8.8,0) {$w$};
        \node (v) at (10.8,0) {$v$};
        \draw (w) to node[below]{\fns{$\Sim$}} (v);
        \draw[densely dotted, bend right=35] (v) to node[above]{\fns{$\Sim$}} (w);
        \node at (9.8,-.85) {$\Sim_{\classneg}$};
        \node at (12,.75) {$+$};
        \node (w) at (13.2,0) {$w$};
        \node (v) at (15.2,0) {$v$};
        \node (s) at (13.2,1.5) {$s$};
        \node (t) at (15.2,1.5) {$t$};
        \draw (w) to node[below]{\fns{$\Sim$}} (v); 
        \draw[densely dotted] (w) to node[left]{\fns{$R$}} (s);
        \draw (v) to node[right]{\fns{$R$}} (t);
        \draw[densely dotted] (s) to node[above]{\fns{$\Sim$}} (t);
        \node at (14.2,-.85) {$\Sim_{\Box}$};
    \end{tikzpicture}
  \end{center}
  (Here the \guillemotleft$+$\guillemotright\ binds stronger than the \guillemotleft or\guillemotright.)
  The dashed lines in Figure~\ref{fig:sim} indicate the gluing conditions
  for the simulation conditions of $\wfrown, \vsmile$ and $\vfrown$.
  A similar gluing condition appears in the bisimulations for the non-contingency
  modality~\cite{FanWanDit14}, stating that any two successors of $w$
  are connected by the simulation:
  
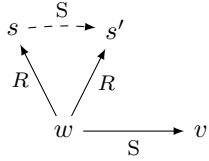
\begin{center}
    \begin{tikzpicture}[scale=.9, arrows=-latex]
        \node (w) at (0,0) {$w$};
        \node (v) at (2,0) {$v$};
        \node (s) at (-.75,1.5) {$s$};
        \node (sp) at (.75,1.5) {$s'$};
        \draw (w) to node[below]{\fns{$\Sim$}} (v); 
        \draw[dashed, bend left=5] (s) to node[above]{\fns{$\Sim$}} (sp);
        \draw (w) to node[left]{\fns{$R$}} (s);
        \draw (w) to node[right]{\fns{$R$}} (sp);
    \end{tikzpicture}
  \end{center}

  This suggests a general approach towards constructing (bi)simulations for combined
  modalities. The main challenge in this line of research would be finding
  a suitable gluing condition for each combined modality.
  The simulations studied in the present paper provide examples that may help
  towards discovering this general pattern, if it exists.%

\section{Conclusion}\label{sec:conc}

\noindent
In this work, we investigated simulations for logics equipped with subclassical negations and their associated restoration modalities. Parametric in a restorative similarity type, we developed appropriate notions of simulation, and established both an intrinsic (Hennessy--Milner-type) and a relative (van Benthem-type) characterization theorem.
These results were obtained for a range of known as well as novel modal logics.
The findings pave the way for various avenues of further research, including the following:

\paragraph{Combined modalities.}
  While our approach already encompasses a wide variety of logics, in 
  Section~\ref{sec:combined} we speculated about an even more general perspective on simulations and their associated characterization results.
  Pursuing this line of research requires a deeper understanding of the underlying pattern of \guillemotleft gluing conditions\guillemotright.
  The concrete instances examined in the present paper may serve as useful case studies in uncovering this pattern.
  It would also be interesting to investigate the precise relation between our gluing-based construction and the uniform construction of bisimulations for bundled modalities developed in~\cite{DingYang2026}, especially in settings without classical negation and with several bundled operators in the same language. 
  Taking this one step further, one could even aim to extend the analysis to
  combined modalities constructed from repeated use of $||$ and $\&$.

\paragraph{Relation to coalgebraic logic.}
  In the setting of classical coalgebraic logic, i.e.~coalgebraic logic
  used to model modal extensions of classical propositional logic,
  a generic notion of bisimulation for any modal signature was
  devised in~\cite{BakHan17}. 
%
  In the special case of non-contingency logic, the resulting notion was shown to be closely aligned with the existing one, studied in~\cite{FanWanDit14,BakDitHan17}.
  It would be worthwhile to investigate how such bisimulations
  behave in the presence of restoration modalities.

  A generalization of the work in~\cite{BakHan17} 
  beyond classical coalgebraic logic was proposed in~\cite{GroHanKur20}. However, this extension does not always yield an intrinsic characterization result for modal extensions of positive logic
  (see~\cite[Remark~4.6]{GroHanKur20}).
  Starting instead with a dual adjunction between the category
  of sets and the category of distributive lattices
  ---the latter being the setting adopted in the present paper--- may provide a solution to this issue.

\paragraph{More on restorative modal logics.}
  In an orthogonal direction, further model-theoretic results could be pursued, such as a finite-model property or Sahlqvist-style theo\-rems for restoration modalities. For the latter, one could draw on Sahlqvist's original result~\cite{Sah75} and its topological proof~\cite{SamVac89}, as well as on its adaptation to positive modal logic~\cite{CelJan99}.

\footnotesize
\bibliographystyle{plainnat}
\bibliography{references.bib}

\end{document}